\documentclass[1 [leqno,11pt]{amsart}
\usepackage{amssymb, amsmath}

\allowdisplaybreaks[4]
\usepackage{color}
\usepackage[plainpages=false,pdfpagelabels]{hyperref}
\newtheorem{theo}{Theorem}[section]
\newtheorem{lemma}{Lemma}[section]
\newtheorem{prop}{Proposition}[section]

\newtheorem {remark} {Remark} [section]

\numberwithin{equation}{section}

\makeatletter
\def\sommaire{\@restonecolfalse\if@twocolumn\@restonecoltrue\onecolumn
\fi\chapter*{Sommaire\@mkboth{SOMMAIRE}{SOMMAIRE}}
  \@starttoc{toc}\if@restonecol\twocolumn\fi}
\makeatother

\makeatletter
\def\thebibliographie#1{\chapter*{Bibliographie\@mkboth
  {BIBLIOGRAPHIE}{BIBLIOGRAPHIE}}\list
  {[\arabic{enumi}]}{\settowidth\labelwidth{[#1]}\leftmargin\labelwidth
  \advance\leftmargin\labelsep
  \usecounter{enumi}}
  \def\newblock{\hskip .11em plus .33em minus .07em}
  \sloppy\clubpenalty4000\widowpenalty4000
  \sfcode`\.=1000\relax}

\makeatother

\makeatletter
\def\references#1{\section*{R\'ef\'erences\@mkboth
  {R\'EF\'ERENCES}{R\'EF\'ERENCES}}\list
  {[\arabic{enumi}]}{\settowidth\labelwidth{[#1]}\leftmargin\labelwidth
  \advance\leftmargin\labelsep
  \usecounter{enumi}}
  \def\newblock{\hskip .11em plus .33em minus .07em}
  \sloppy\clubpenalty4000\widowpenalty4000
  \sfcode`\.=1000\relax}

\makeatother

\def\inte#1{
\displaystyle\mathop{#1\kern0pt}^\circ
}

\newcommand{\beq}{\begin{equation}}
\newcommand{\eeq}{\end{equation}}
\newcommand{\ben}{\begin{eqnarray}}
\newcommand{\een}{\end{eqnarray}}
\newcommand{\beno}{\begin{eqnarray*}}
\newcommand{\eeno}{\end{eqnarray*}}

\let\al=\alpha

\let\d=\delta
\let\e=\varepsilon

\let\r=\rho

\let\f=\phi

\let\wh=\widehat

\def\cK{{\mathcal K}}

\def\cR{{\mathcal R}}

\def\virgp{\raise 2pt\hbox{,}}
\def\cdotpv{\raise 2pt\hbox{;}}

\def\eqdefa{\buildrel\hbox{\footnotesize def}\over =}

\def\im {\mathop{\rm Im}\nolimits}

\def\C{\mathop{\mathbb C\kern 0pt}\nolimits}
\def\DD{\mathop{\mathbb D\kern 0pt}\nolimits}
\def\EE{\mathop{\mathbb E\kern 0pt}\nolimits}
\def\K{\mathop{\mathbb K\kern 0pt}\nolimits}
\def\N{\mathop{\mathbb  N\kern 0pt}\nolimits}
\def\Q{\mathop{\mathbb  Q\kern 0pt}\nolimits}
\def\R{{\mathop{\mathbb R\kern 0pt}\nolimits}}
\def\SS{\mathop{\mathbb  S\kern 0pt}\nolimits}
\def\St{\mathop{\mathbb  S\kern 0pt}\nolimits}
\def\Z{\mathop{\mathbb  Z\kern 0pt}\nolimits}
\def\ZZ{{\mathop{\mathbb  Z\kern 0pt}\nolimits}}
\def\H{{\mathop{{\mathbb  H\kern 0pt}}\nolimits}}
\def\PP{\mathop{\mathbb P\kern 0pt}\nolimits}
\def\TT{\mathop{\mathbb T\kern 0pt}\nolimits}

\def\pa{\partial}

\newcommand{\la}{\lambda}

\newcommand{\andf}{\quad\hbox{and}\quad}
\newcommand{\with}{\quad\hbox{with}\quad}

\def\pr{\psi^\hbar}

\def\im{{\rm i}}

\def\bg{{\vec g}}
\def\bv{{\vec v}}

\begin{document}

\title[Global controllability of the semiclassical NLS] {On the global approximate controllability  in small time of semiclassical 1-D Schrödinger equations between two states with positive quantum densities}

\author[J. M. Coron]{Jean-Michel  Coron}
\address[J.-M. Coron]%
{Sorbonne Universit\'{e}, Universit\'{e} Paris-Diderot SPC, CNRS, INRIA, Laboratoire Jacques-Louis Lions, LJLL,  \'{e}quipe CAGE, F-75005 Paris.}
\email{coron@ann.jussieu.fr}
\author[S. Xiang]{Shengquan Xiang}
\address[S. Xiang]%
{B\^{a}timent des Math\'{e}matiques, EPFL, Station 8, CH-1015 Lausanne, Switzerland.}
\email{shengquan.xiang@epfl.ch}
\author[P. ZHANG]{Ping Zhang}%
\address[P. Zhang]
{Academy of
Mathematics $\&$ Systems Science and  Hua Loo-Keng Key Laboratory of
Mathematics, The Chinese Academy of Sciences, China, and School of Mathematical Sciences, University of Chinese Academy of Sciences, Beijing 100049, China.}
\email{zp@amss.ac.cn}
\subjclass[2010]{}
\begin{abstract} In this paper, we study, in the semiclassical  sense,  the global approximate controllability  in small time of the quantum density and quantum momentum of the 1-D semiclassical cubic Schrödinger equation with two controls between two states with positive quantum densities. We first control the asymptotic expansions of the zeroth
and first order of the physical observables via Agrachev-Sarychev's method.
 Then we conclude the proof  through techniques of semiclassical
approximation of the nonlinear Schr\"{o}dinger equation.
 \end{abstract}

\maketitle


\setcounter{equation}{0}

\section{Introduction}
The purpose of this paper is to investigate the global approximate controllability of the physical observables of the
 following one-dimensional cubic Schr\"{o}dinger equation in the semiclassical regime
\begin{equation}\label{S1eq1}
\left\{\begin{array}{l}
\displaystyle i\hbar\pa_t \pr= -\frac{\hbar^2}2\pa_x^2\pr+\left(F^\hbar+(|\pr|^2-1)\right)\pr,  \qquad (t,x)\in\R^+\times{\Bbb T} \, ,
 \\
\displaystyle  \pr|_{t=0}=\pr_0=a^0(x;\hbar)\exp\left({i}S(x)/{\hbar}\right) \, ,
\end{array}\right.
\end{equation}
where $\pr$ denotes the wave function in quantum mechanics so that $\int_{\TT}|\pr(t, x)|^2\,dx=1$, $\hbar$ denotes
the Planck constant, and $\Bbb T:=\mathbb{R}/2\pi \mathbb{Z}$. The function $\pr$ is complex-valued while the functions
 $S(x)$  and $F^{\hbar}$ are real-valued.  The potential $V\eqdefa F^\hbar+(|\pr|^2-1)$  is  composed by two parts: $(|\pr|^2-1)$ corresponds to the Gross-Pitaevskii equations, and  $F^\hbar$ denotes a background charge that can be controlled in semiconductor applications.  In control theory, typically for the case $\hbar=1$,  the term $F^\hbar \pr$ is called bilinear control.

We  assume that, for some positive integer $k$, 
\begin{equation}\label{cond-reg-be}
\begin{split}
S(x)\in H^{k+3},\ \ a^0(x;\hbar)&=a^0_0(x)+ \hbar a^0_1(x)+\hbar R(x; \hbar)\with\\
  a^0_0(x)\in H^{k+2},\ \ a^0_0(x)>0, &\ \ a^0_1(x)\in H^k\   \andf \lim_{\hbar\rightarrow 0^+} \lVert R(x; \hbar) \lVert_{H^k}=0.
\end{split}
\end{equation}

 For \eqref{S1eq1},
the well-known Madelung transform (see \cite{1927-Madelung-ZP}) introduces
two real variables $\rho^\hbar (t, x)$ and $S^\hbar (t, x)$ such that $\psi^\hbar =
\sqrt{\rho^\hbar}\exp\left({i}S^\hbar/{\hbar}\right).$ Then we can equivalently rewrite \eqref{S1eq1} as
\begin{equation}\label{S1eq1a}
\left\{\begin{array}{l}
\displaystyle \pa_t\rho^\hbar+\pa_x\left(\rho^\hbar u^\hbar\right)=0 \qquad (t,x)\in\R^+\times{\Bbb T} \, ,
 \\
 \displaystyle \pa_t\left(\rho^\hbar u^\hbar\right)+\pa_x\left(\rho^\hbar(u^\hbar)^2\right)+\frac12\pa_x(\rho^\hbar)^2+\pa_x\left(F^\hbar\right)\rho^\hbar=\frac{\hbar^2}2\rho^\hbar
 \Bigl(\frac{\pa_x^2\sqrt{\rho^\hbar}}{\sqrt{\rho^\hbar}}\Bigr)_x,\\
\displaystyle  \rho^\hbar|_{t=0}=|a^0(x,\hbar)|^2,\quad u^\hbar|_{t=0}=\pa_xS(x),
\end{array}\right.
\end{equation}
with
\begin{equation}
u^\hbar (t, x)\eqdefa \pa_x S^\hbar (t, x).
\end{equation}
The right-hand side of the second equation of \eqref{S1eq1a} is called the quantum pressure.

 Formally, by taking $\hbar \to 0$ in \eqref{S1eq1a}, we
 obtain the following compressible Euler equation in $\R^+\times{\Bbb T}$ for  quantities $(\r_0, u_0)$ with an additional force $\eta_0(t, x)$:
\begin{equation}\label{S1eq6a}
\left\{\begin{array}{l}
\displaystyle  \pa_t\r_0+\pa_x(u_0\r_0)=0,\\
\displaystyle \pa_t u_0+u_0\pa_xu_0+\pa_x\r_0= \eta_0,
 \\
\displaystyle (\r_0,u_0)|_{t=0}=(|a^0_0|^2,\pa_xS),
\end{array}\right.
\end{equation}
where
\begin{equation}
\eta_0 \eqdefa -\pa_x F_0,
\end{equation}
and where the quantities $\rho_0, u_0, F_0$ correspond to the leading order of $\rho^{\hbar}, u^{\hbar}, F^{\hbar}$ under WKB sense.
We refer to
Section \ref{sec-sub-gset}, especially to  Equation \eqref{S1eq5}--\eqref{S1eq6b}  for details of this expansion and even for the related higher order expansions.\\

Since $u_0$ is the derivative of some function $S_0$ on the torus ${\Bbb T}$, it satisfies $\int_{\mathbb{T}} u_0\,dx=0$.
 For this purpose, we denote by $H^s_0$  the codimension one  subspace of $H^s$ such that its elements satisfy $\int_{\mathbb{T}} f\,dx=0$.

Corresponding to the Newtonian mechanics, we call $\rho^\hbar$ and $\rho^\hbar u^\hbar$ the quantum density and quantum momentum respectively.
In the whole space case and if there is no superfluid at infinity and  $f'(0)>0$, Grenier  \cite{1998-Grenier-PAMS} solved the limit
problem before the formation of singularity in the limit system  with initial data in Sobolev spaces. The main
idea in  \cite{1998-Grenier-PAMS} is to use the symmetrizer of the limit system \eqref{S1eq6a} to get $H^s$ energy
estimates which are uniform in $\e$ for a singularly perturbed system. Nevertheless  this method does not work for the semiclassical limit
of  Schr\"odinger-Poisson equations, as the resulting limit
system is not a symmetric hyperbolic one.   The third author of this manuscript \cite{2002-Zhang-SIAMJMA}  used the Wigner measure approach  to study the semiclassical
limit of Schr\"odinger-Poisson equation (see  \cite{2002-Zhang-JPDE} for more general nonlinearity).
 The assumption that $f'(0)>0$ in \cite{2002-Zhang-JPDE} was removed by Alazard and Carles in \cite{AC09}. With  superfluid at infinity to pass over an obstacle, the
corresponding problem was solved by Lin and Zhang in \cite{2006-Lin-Zhang-ARMA}.\\

The goal of this paper is to find control $F^\hbar$
\begin{equation}
F^\hbar(t,x)=F_0(t, x)+ \hbar F_1(t, x) \textrm{ with }  F_i(t, x)\in C([0, T]; E),
\end{equation}
where
\begin{equation}
E\eqdefa {\rm span}\{\sin x, \cos x\} \textrm{ is a two-dimensional control space }
\end{equation}
 such that we are able  to get
 the approximate controllability of the physical observables $\rho^\hbar$ and $u^\hbar.$

More precisely, our main result states as follows.

\begin{theo}\label{th1.1}
{\sl Let $k\geq 3, T>0$ and $\varepsilon>0$. Let $(\r_0^f, u_0^f)\in H^k\times H^{k}_0$ and $(\r_1^f, u_1^f)\in H^{k-2}\times H^{k-2}_0$, which satisfy
\begin{equation}
\int_{\mathbb{T}} \r_0^f\,dx= \int_{\mathbb{T}}  |a_0^0|^2\,dx,\quad
\int_{\mathbb{T}} \r_1^f\,dx= \int_{\mathbb{T}} 2 \textrm{Re} (\bar{a}_0^0 a_1^0)\,dx \andf \r^f_0>0.  \notag
\end{equation}
 There exist  $(\hat\r_0, \hat u_0)\in H^{k+2}\times H^{k+2}_0, (\hat\r_1, \hat u_1)\in H^k\times H^{k}_0$   which satisfy
\begin{gather}
\int_{\mathbb{T}} \hat \r_0\,dx= \int_{\mathbb{T}}  \r_0^f\,dx,\quad \int_{\mathbb{T}} \hat \r_1\,dx= \int_{\mathbb{T}} \r_1^f\,dx \andf
\hat \r_0>0,  \notag\\
\lVert (\hat\r_0, \hat u_0)- (\r_0^f, u_0^f)\lVert_{H^k\times H^{k}}\leq \varepsilon, \quad \lVert (\hat\r_1, \hat u_1)- (\r_1^f, u_1^f)\lVert_{H^{k-2}\times H^{k-2}}< \varepsilon,\notag
\end{gather}
and controls $F^{\hbar}$ in $C^{\infty}([0, T]; E)$ such that,
the solution of equation \eqref{S1eq1} satisfies
\begin{align}
\psi^{\hbar}(T)=& \big(\hat a_0+ \hbar \hat a_1+ \hbar r^{\hbar}_a \big) \exp \Big(\frac{i}{\hbar} \big(\hat S_0+ \hbar \hat S_1+ \hbar r^{\hbar}_S\big)\Big)  \notag\\
=&\big(\hat a_0 e^{i \hat S_1}+ s^{\hbar}\big) \exp\big(\frac{i}{\hbar}\hat S_0\big),  \notag
\end{align}
with
\begin{gather}
|\hat a_0(x)|^2= \hat\r_0(x), \quad 2{\rm Re}(\bar{\hat{a}}_0 \hat a_1)= \hat\r_1(x),     \notag\\
\pa_x \hat S_0(x)= \hat u_0(x),\quad
\pa_x \hat S_1(x)= \hat u_1(x),    \notag\\
\lim_{\hbar\rightarrow 0^+} \lVert r^{\hbar}_a \lVert_{H^{k-2}}=0, \quad \lim_{\hbar\rightarrow 0^+} \lVert r^{\hbar}_S \lVert_{H^{k-1}}=0, \quad  \lim_{\hbar\rightarrow 0^+} \lVert s^{\hbar} \lVert_{H^{k-2}}=0.   \notag
\end{gather}}
\end{theo}

\begin{remark}
The proof of Theorem \ref{th1.1} can in fact ensure the same result
 for  nonlinear Schr\"{o}dinger equations with more general nonlinearity:  $f(|\pr|^2)\pr$ which satisfies $f'(0)>0$.
\end{remark}
\begin{remark}
Let $N\geq 2$.  Using the same approach, we are able to   find control
\beno
F^\hbar(t,x)=\sum_{i=0}^N \hbar^i F_i(t, x)\with  F_i\in C([0, T]; E), \ i=1,\cdots,N, \eeno
 so that we can get the approximate controllability of the physical observables up to order $N.$
\end{remark}
\begin{remark}
Along the same line to the proof of Theorem \ref{th1.1}, we can work out  the same type of result in any space dimension, i.e. on $\mathbb{T}^d$ for any $d\in \mathbb{N}\setminus\{0\}$,  with controls be restricted to some finite dimensional space
\begin{equation}
\mathcal{E}:= span\{\sin (\vec l\cdot \vec x), \cos (\vec l\cdot \vec x); \; \vec  l\in \mathcal{K}_d\}.
\end{equation}
An example is
\begin{equation}
\mathcal{K}_d:= \{(l_1, ..., l_d); \;  l_j= 0 \textrm{ or }\pm 1\},
\end{equation}
which also has been used by Duca and Nersesyan \cite{Nersesyan-2021}  for the study of  approximate controllability between any pairs of eigenstates of NLS in $\mathbb{T}^d$.

\end{remark}
\begin{remark}
The study of the controllability of the Schr\"{o}dinger equation corresponds to the case that $\hbar=1$ in \eqref{S1eq1} has been attracted by many authors. Let us mention, in particular, the survey paper \cite{Laurent-2014} by Laurent and the references therein, \cite{Lebeau-1992} by Lebeau for  internal controllability of the multidimensional Schr\"{o}dinger equation using microlocal analysis. The local exact controllability of the one dimensional bilinear Schr\"{o}dinger equation was first proved by Beauchard \cite{Beauchard-2005-JMPA9} using a  Nash--Moser iteration process, and then simplified by Beauchard-Laurent in \cite{Beauchard-Laurent-2010}. Finally, in \cite{2012-Sarychev-MCRF} Sarychev also studied the approximate controllability of NLS in 1D and 2D torus based on Agrachev--Sarychev method, and again relying on this method  Duca and Nersesyan \cite{Nersesyan-2021} have studied the approximate controllability between any pairs of eigenstates of NLS in $\mathbb{T}^d$.
\end{remark}

\medskip

\setcounter{equation}{0}
\section{Sketch of the proof}

\subsection{Semiclassical approximation in the general setting}\label{sec-sub-gset}
Let us recall from \cite{1998-Grenier-PAMS} that instead of looking, as usual, for solutions $\pr$ to \eqref{S1eq1} of the form
\beno
\pr(t,x)=a^\hbar(t,x)\exp\left({i}S(t,x)/{\hbar}\right)
\eeno
with $S(t,x)$ independent of $\hbar$ and $a^\hbar(t,x)$ a real-valued function, we can look for solutions $\pr$ of \eqref{S1eq1} with the form:
\beq \label{S1eq4}
\begin{split}
\pr(t,x)=&a^\hbar(t,x)\exp\left(i S^\hbar(t,x)/\hbar\right)\with\\
a^\hbar(t,x)=& a_0(t, x)+ \hbar a_1(t, x)+ \hbar r_a^{\hbar}(t, x,) \\
 S^\hbar(t,x)=& S_0(t, x)+ \hbar S_1(t, x)+ \hbar r_S^{\hbar}(t, x),
\end{split}
\eeq where $a^\hbar$ is a complex-valued function.
By inserting \eqref{S1eq4} into \eqref{S1eq1}, we obtain
\begin{equation}\label{S1eq5}
\left\{\begin{array}{l}
\displaystyle \pa_t S^\hbar+\frac12\left(\pa_xS^\hbar\right)^2+F^\hbar+|a^\hbar|^2-1=0,
 \\
\displaystyle  \pa_ta^\hbar+\pa_xS^\hbar\pa_xa^\hbar+\frac12a^\hbar\pa_x^2S^\hbar=i\frac\hbar2\pa_x^2a^\hbar,\\
\displaystyle S^\hbar|_{t=0}=S(x),\quad a^\hbar|_{t=0}=a^0(x;\hbar).
\end{array}\right.
\end{equation}
Let us define
\begin{equation}
\eta_j(t, x)\eqdefa  -\pa_x F_j(t, x) \in E.  \notag
\end{equation}
Since $F_j\in E$ we can determine $F_j$ from $\eta_j$ by $$F_j (t,x)= -\int_0^x \eta_j (t, x) dx+(1/2\pi)\int_0^{2\pi}(2\pi-s)\eta_j(s)ds.$$
In view
  of \eqref{S1eq6a}, by comparing the coefficients of $\hbar^j$ on both sides of \eqref{S1eq5}, at least formally, we are able to obtain equations for the coefficients of the expansions in \eqref{S1eq4}.  The $\hbar^0$ order equation reads as
\begin{equation}\label{S1eq6}
\left\{\begin{array}{l}
\displaystyle \pa_t S_0+\frac12\left(\pa_xS_0\right)^2+|a_0|^2-1=-F_0(t, x),
 \\
\displaystyle  \pa_ta_0+\pa_xS_0\pa_xa_0+\frac12a_0\pa_x^2S_0=0,\\
\displaystyle S_0|_{t=0}=S(x),\quad a_0|_{t=0}=a^0_0(x).
\end{array}\right.
\end{equation}

Let us denote $\rho_0\eqdefa |a_0|^2$ and $u_0\eqdefa \pa_xS_0.$ We recover the compressible Euler system \eqref{S1eq6a}.
However,  \eqref{S1eq6} is not equivalent to \eqref{S1eq6a} though it is derived from the previous one.
 In fact, let us denote $a_0$ as $ a_0^{\rm r}+ia_0^\im$,  then \eqref{S1eq6} can be  reformulated as the following real-valued equation:
\begin{equation}\label{S1eq6b}
\left\{\begin{array}{l}
\displaystyle \pa_ta_0^{\rm r}+u_0\pa_xa_0^{\rm r}+\frac12a_0^{\rm r}\pa_xu_0=0,\\
\displaystyle \pa_ta_0^{\rm i}+u_0\pa_xa_0^{\rm i}+\frac12a_0^{\rm i}\pa_xu_0=0,\\
\displaystyle \pa_t u_0+u_0\pa_xu_0+\pa_x\r_0=\eta_0(t, x),
 \\
\displaystyle (a_0^{\rm r},a_0^{\rm i},u_0)|_{t=0}=({\rm Re}(a_0^0), {\rm Im}(a_0^0),\pa_xS).
\end{array}\right.
\end{equation}

 For $\hbar^1$ order of \eqref{S1eq5}, we get the following equation for $(a_1,S_1)$:
\begin{equation}\label{S1eq7}
\left\{\begin{array}{l}
\displaystyle \pa_t S_1+\pa_xS_0\pa_xS_1+2{\rm Re}\left(a_0\bar{a}_1\right)=-F_1(t, x),
 \\
\displaystyle  \pa_ta_1+\pa_xS_0\pa_xa_1+\pa_xS_1\pa_xa_0+\frac12\left(a_0\pa_x^2S_1+a_1\pa_x^2S_0\right)=\frac{i}2\pa_x^2a_0,\\
\displaystyle S_1|_{t=0}=0,\quad a_1|_{t=0}=a^0_1(x).
\end{array}\right.
\end{equation}
Again by  defining $a_1\eqdefa a_1^{\rm r}+ia_1^\im$ and $u_1\eqdefa \pa_xS_1$ we reformulate \eqref{S1eq7} as
\begin{equation}\label{S1eq7b}
\left\{\begin{array}{l}
\displaystyle \pa_ta_1^{\rm r}+u_0\pa_xa_1^{\rm r}+u_1\pa_xa_0^{\rm r}+\frac12\bigl(a_0^{\rm r}\pa_xu_1+a_1^{\rm r}\pa_xu_0\bigr)=-\frac12\pa_x^2a_0^{\rm i},\\
\displaystyle \pa_ta_1^{\rm i}+u_0\pa_xa_1^{\rm i}+u_1\pa_xa_0^{\rm i}+\frac12\bigl(a_0^{\rm i}\pa_xu_1+a_1^{\rm i}\pa_xu_0\bigr)=\frac12\pa_x^2a_0^{\rm r},\\
\displaystyle \pa_t u_1+u_0\pa_xu_1+u_1\pa_xu_0+2\pa_x\bigl(a_0^{\rm r}a_1^{\rm r}+a_0^{\rm i}a_1^{\rm i}\bigr)=\eta_1(t, x),
 \\
\displaystyle (a_1^{\rm r},a_1^{\rm i},u_1)|_{t=0}=({\rm Re}(a_1^0), {\rm Im}(a_1^0),0).
\end{array}\right.
\end{equation}

\subsection{On the control of physical observables}
Following the idea of \cite{1998-Grenier-PAMS}, for any time $T>0,$ it sounds interesting to let  the final state at time $T$ to
 approximate to $\bar a \exp (i \bar S/\hbar)$ with the help of some extra force as control, where $\bar a$ and $\bar S$ are target functions. Nevertheless, we are not able to control both $a^r_0$ and $a^i_0$ simultaneously.  Indeed it is easy to observe from the zeroth-order approximate
  system   \eqref{S1eq6b} that if ${\rm Im}(a^0_0)=0$, we  find that $a^i_0$ is always zero in the future.

On the other hand,
it is of physical importance and is natural to consider the controllability of the physical observables $\rho^\hbar$ and $u^\hbar.$ Hence, first
we are going to control the compressible Euler system \eqref{S1eq6a} through the control $\eta_0.$
 As for the first order approximation, we  observe from \eqref{S1eq7b} that
\beno
\begin{split}
\pa_t\left(\bar{a}_0a_1\right)+u_0\pa_x\left(\bar{a}_0a_1\right)+u_1\bar{a}_0\pa_x a_0+\frac12\r_0\pa_xu_1+\bar{a}_0a_1\pa_xu_0=\frac{i}2\bar{a}_0\pa_x^2a_0.
\end{split}
\eeno
Taking the real part of the above equation and defining  $\rho_1\eqdefa 2{\rm Re}(\bar{a}_0a_1)$ give rise to
\beq\label{S1eq24}
\pa_t\r_1+\pa_x\left(u_0\r_1+u_1\r_0\right)=\frac{i}2\pa_x\left(\bar{a}_0\pa_xa_0-a_0\pa_x\bar{a}_0\right).
\eeq
This leads to the following control system:
\begin{equation}\label{S1eq25}
\left\{\begin{array}{l}
\displaystyle \pa_t\r_1+\pa_x\left(u_0\r_1+u_1\r_0\right)=\frac{i}2\pa_x\left(\bar{a}_0\pa_xa_0-a_0\pa_x\bar{a}_0\right)
 \\
\displaystyle  \pa_tu_1+\pa_x\left(u_0u_1\right)+\pa_x\r_1=\eta_1.
\end{array}\right.
\end{equation}
As we mentioned before the compressible Euler system \eqref{S1eq6a} is not equivalent to the zeroth-order approximate system \eqref{S1eq6b}. Hence
in order to determine the right-hand side of the first equation of \eqref{S1eq25},  let us
define
\begin{equation}
A\eqdefa \frac{i}2\left(\bar{a}_0\pa_xa_0-a_0\pa_x\bar{a}_0\right).
\end{equation}
  Then it follows from \eqref{S1eq6b} that
  \begin{gather}\label{S1eq25a}
  \begin{cases}
  \pa_tA+u_0\pa_xA+2\pa_xu_0A=0,\\
  A(0)=A_0\eqdefa a_0^\im(0)\pa_xa_0^{\rm r}(0)-a_0^{\rm r}\pa_xa_0^\im(0).
  \end{cases}
  \end{gather}

Once $(\rho_0,u_0)$ is determined from \eqref{S1eq6a}, we can solve for $\left(a_0^{\rm r}, a_0^\im\right)$ via \eqref{S1eq6b}. Similarly, with
$(\rho_1,u_1)$ being determined, we can solve for $\left(a_1^{\rm r}, a_1^\im\right)$ via \eqref{S1eq7b}. Along the same line,
we can control the physical observables of any $k$-th order, and then solve for $\left(a_k^{\rm r}, a_k^\im\right).$  For a concise
presentation, we shall not pursue in this direction here.

Then the proof of Theorem \ref{th1.1} will be split in  two parts: the
controllability of the asymptotic expansions of the  physical quantities, and then the justification of the semiclassical approximation to the system \eqref{S1eq1}.

 By combining  the systems \eqref{S1eq6a}, \eqref{S1eq25} and \eqref{S1eq25a}, we obtain the following real valued controlled system in $[0,T]\times\TT$
 for the first two terms in the asymptotic expansions of physical quantities:
\begin{equation}\label{S2eq1}
\left\{\begin{array}{l}
\displaystyle  \pa_t\r_0+\pa_x(u_0\r_0)=0,\\
\displaystyle \pa_t u_0+u_0\pa_xu_0+\pa_x\r_0=\eta_0,\\
\displaystyle \pa_tA+u_0\pa_xA+2\pa_xu_0A=0,\\
\displaystyle \pa_t\r_1+\pa_x\left(u_0\r_1+u_1\r_0\right)=\pa_xA,
 \\
 \displaystyle  \pa_tu_1+\pa_x\left(u_0u_1\right)+\pa_x\r_1=\eta_1.
\end{array}\right.
\end{equation}
Let $\vec{\r}\eqdefa \begin{pmatrix}
\r_0\\
\r_1
\end{pmatrix}$ and $\vec{u}\eqdefa \begin{pmatrix}
u_0\\
u_1
\end{pmatrix}.$
Given initial state $\vec{\r}(0)=\vec{g}\eqdefa \begin{pmatrix}
g_0\\
g_1
\end{pmatrix}$ and $\vec{u}(0)=\vec{v}\eqdefa \begin{pmatrix}
v_0\\
v_1
\end{pmatrix},$  $A_0$ and the state at time $T$ that $\vec{\r}(T)=\wh{\vec{g}}\eqdefa \begin{pmatrix}
\wh{g}_0\\
\wh{g}_1
\end{pmatrix}$  and $\vec{u}(T)=\wh{\vec{v}}\eqdefa \begin{pmatrix}
\wh{v}_0\\
\wh{v}_1
\end{pmatrix},$  we are going to search for a control  $\vec{\eta}\eqdefa \begin{pmatrix}
\eta_0\\
\eta_1
\end{pmatrix}$ such that the solution of  system \eqref{S2eq1} satisfying the initial condition
\begin{equation}\label{eq-init}
\vec{\r}(0)=\vec{g}, \, \vec{u}(0)=\vec{v}
\end{equation}
and $(u_0, u_1, \r_0, \r_1)(T)$ is closed to $(\hat{v}_0, \hat{v}_1, \hat{g}_0, \hat{g}_1)$.

Let us remark here that we will not control the quantity $A$, though it  will be used to control the next order terms.

\subsection{Approximate controllability of the physical observables}
Therefore, we are going to control the limit system \eqref{S2eq1} for $(\vec u, \vec \rho, A)$. It is a complex system with cascade structure: $$(\rho_0, u_0)\rightarrow A\rightarrow (\rho_1, u_1).$$
\begin{itemize}
\item Firstly, $(\rho_0, u_0)$ is governed by the Saint-Venant equation with finite dimensional distributed control $\eta$ that acts on the ``velocity" part  $u_0$.  Hence, it is rather natural to expect controllability properties of this part.  Actually, in \cite{2011-Nersisyan-CPDE} Nersisyan has considered the approximate controllability of the  3D compressible Euler equations which is similar to this Saint-Venant equation, and we mimic his approach on the controllability of this part. \\
We comment here, as we shall see later on in Section  \ref{sec-cont-lim}, thanks to the ``transport" natur of the equation on $\rho_0$, during the whole controlling process the value of $\rho_0$ is always sufficiently close to some given trajectory that is strictly positive.
\item Next, $A$ verifies a transport equation that is influenced by the value of $u_0$. However, there is no direct control mechanism on this quantity: therefore it is indirectly controlled by $\eta_0$. \\
Moreover, we observe from the preceding section that the quantities $\rho_0, \rho_1$ and $A$  all come from $a^{\hbar}$. Hence it is reasonable to expect that they will share certain similarities. As it is actually directly  shown in Equation \eqref{S2eq1} that they all satisfy ``transport" type equations. Consequently, similar to $\rho_0$, the value of $A$ is always sufficiently close to some given trajectory. \\
Finally, we remark here that we are not going to control $A$ to some given final state $\tilde{A}$. Indeed, by looking at \eqref{S1eq6b} for $(a_0^r, a_0^i),$ we observe that both terms satisfy the same equation, as a  consequence, we are not able to control both items. If we further look at the definition of  $\rho_0$ and $A$, it becomes reasonable that we control $\rho_0$ as the ``density" part of $a_0$ and ``lose the control" of the ``rotation" part $A$.    But instead we will control the trajectory of $A(t)$ by making it to stay close to a given one. The value of $A(t)$ will be used to control $(\rho_1, u_1)$.
\item Finally, $(\rho_1, u_1)$ is a coupled systems with finite dimensional controlling terms that is also influenced by the value of $(\rho_0, u_0, A)$.  The main feature on the controllability of this part, different from well-posedness issue which is interested in the uniqueness and  existence issues, is that the value of $(\rho_0, u_0, A)$ also plays a significant role in the controlling process. Therefore, it should be regarded as a bilinear system and the influence of the  control term $\eta_0$ should not be ignored. This is actually one of the main novelties on the controllability part of this paper. \\
Similar to $\rho_0$ and $A$, the value of $\rho_1$ is also kept close to a constructed trajectory.
\end{itemize}

The main result lists as follows, while the detailed proof of which will be presented  in Section \ref{sec-cont-lim}.
\begin{theo}\label{thm-con-lim-main}
{\sl Let $k\geq 3, T>0,$ $\vec{g}\in H^{k+2}\times H^{k},$ $ \vec{v} \in H^{k+2}_0\times H^{k}_0,$ $\wh{\vec{g}}\in H^{k}\times H^{k-2}$  $\wh{\vec{v}}\in H^{k}_0\times H^{k-2}_0$ and
$A_0\in H^{k+1}$,  which satisfy
\beq
\int_{\TT}g_0\,dx=\int_{\TT}\wh{g}_0\,dx=\al_0,\quad \int_{\TT}{g}_1\,dx=\int_{\TT}\wh{g}_1\,dx=\al_1 \andf g_0, \wh{g}_0\geq c_0>0.  \notag
\eeq
Then for any $\e>0,$ there exist controls $\eta_0, \eta_1\in C^{\infty}([0, T]; E)$ such that the solution  $(u_0, u_1, \r_0, \r_1, A)(t)$ to the system \eqref{S2eq1}  satisfies the initial condition \eqref{eq-init} and
\begin{equation}
\lVert (u_0, u_1, \r_0, \r_1)(T)- (\hat{v}_0, \hat{v}_1, \hat{g}_0, \hat{g}_1)\lVert_{H^k\times H^{k-2}\times H^k\times H^{k-2}} \leq \varepsilon.  \notag
\end{equation}
}
\end{theo}

\begin{remark}
One can observe that we have assumed the value of density $g_0, \hat g_0$ being strictly positive in order to  use the general well-posedness theory on symmetrizable hyperbolic systems. As we shall see later on, the constructed approximate solutions $\rho_0^{l}(t), l\in \{0, 1,..., N\}$ is always close to the first constructed solution $\rho_0^{N}(t)$ which is strictly positive.  It will be interesting and also challenging to consider about the case that $g_0, \hat g_0$ are not strictly positive, for example admit finite many points such that $g_0(x)$ equal to zero.
\end{remark}

We remark that the main idea of the proof of Theorem \ref{thm-con-lim-main} is directly motivated by \cite{2010-Nersisyan-COCV,2011-Nersisyan-CPDE} on the controllability of the incompressible and compressible Euler system using the Agrachev--Sarychev method, though we are dealing with a coupled system that is  more complicated.  This approach originates from \cite{Agrachev-Sarychev-2005-JMFM, Agrachev-Sarychev-2006-CMP} by Agrachev--Sarychev for the controllability of 2D Navier--Stokes and 2D Euler equations with finite-dimensional external control, and is further investigated  on many other models, such as  \cite{2012-Sarychev-MCRF, 2006-Shirikyan-CMP} and among others.  Shirikyan \cite{2018-Shirikyan-PAFA}   wrote an instructive introduction on the application of this method upon the viscous Burgers equation.  We refer to Section \ref{sec-cont-lim} on the detailed proof of Theorem \ref{thm-con-lim-main}, where we will also comment on each step  how we get inspired from this method.

Let us also point out  that this method emphasizing geometric control only leads to approximate controllability properties, while exact controllability can be only made on finite dimensional sets.  It is natural to ask whether some local controllability result may further lead to global exact controllability. When dealing with  boundary control or localized internal control  problems,   the exact controllability of the related models can be obtained using  other methods notably the ``return method" and ``Carleman estimates", see for example, \cite{coron, 1996-Coron-Fursikov-RJMP, MR2103189, Xiang-NS-2020} on Navier-Stokes equations, and  \cite{coron-1996-euler-return, 116,  2000-Glass-COCV, Glass-2007-JEMS} on Euler equations and similar Saint-Venant equations. However, there are two major difficulties that prevent us from getting local exact controllability of \eqref{S2eq1} using finite dimensional distributed controls. The first difficulty is related to the fact that  we are only  controlling  four of the five components of the state. The other difficulty is more systemic: it is reasonable to consider the moment theory for the linearized system of this nonlinear problem; however, so far, this method has not been adapted for first order quasi-linear hyperbolic systems. Indeed, a loss of derivative issue appears when we attempt to apply fixed point theorems, see \cite[Chapter 4.2]{coron} for a simplified model that presents the same difficulty. Therefore, it is required to prove exact (partial) controllability of the linearized systems around a class of time-varying trajectories, see for example in \cite[Lemma 17]{116} the author has used this strategy to overcome the loss of derivative issue for Saint--Venant equation with boundary control, for which the proof relies on  the fact that there are enough controllable linear systems  close to the non linear control system (thanks to characteristic lines) to get the local controllability of the nonlinear control system by a suitable fixed point method.   However,  whether this idea can be  adapted to distributed controlled systems by using the moment theory, it  still  remains open.

Finally, in view of the recent progresses in  stabilization problems, see for instance \cite{coron:hal-03161523, SaintVenantPI, Krieger-Xiang-2020}, it also sounds interesting to have a look at the related stabilization problems.

\subsection{Semiclassical limits}
Armed with Theorem  \ref{thm-con-lim-main}, we shall conclude the proof of Theorem \ref{th1.1} by

\begin{theo}\label{thm-sem-main}
{\sl Let $k> 5/2, T>0$. Suppose  that $ (\rho_0, u_0)\in C([0, T]; H^k)$ is the solution of
 \eqref{S1eq6a} determined by Theorem \ref{thm-con-lim-main}, then under the assumptions of Theorem \ref{th1.1},
  for  $\hbar$ small enough, \eqref{S1eq1} has a solution
   of the form \eqref{S1eq4} on the interval $[0, T], $ which  satisfies
\begin{align}
\psi^{\hbar}= a^{\hbar}\exp\big(\frac{i}{\hbar} S^{\hbar}\big)= \big(a_0+ \hbar a_1+ \hbar r_a^{\hbar}\big) \exp \Big(\frac{i}{\hbar} \big(S_0+ \hbar S_1+ \hbar r_S^{\hbar}\big)\Big),  \notag
\end{align}
with
\begin{gather}
  \textrm{$a^{\hbar}$ and $u^{\hbar}$ are uniformly bounded in $C([0, T]; H^k)$ as $\hbar\rightarrow 0^+$}   \notag\\
(a_0, S_0) \textrm{ the solution of } \eqref{S1eq6}, \textrm{ and } a_0, u_0\in C([0, T]; H^{k+2}), \notag\\
(a_1, S_1) \textrm{ the solution of } \eqref{S1eq7}, \textrm{ and }   a_1, u_1\in C([0, T]; H^{k}),   \notag \\
\lim_{\hbar\rightarrow 0^+} \lVert r^{\hbar}_a \lVert_{C([0, T]; H^{k-2})}=0,\; \lim_{\hbar\rightarrow 0^+} \lVert r^{\hbar}_S \lVert_{C([0, T]; H^{k-1})}=0.   \notag
\end{gather}}
\end{theo}

\medskip

\section{The well-posedness and continuous dependence of the limit systems}
This section is devoted to the study of the controlled system \eqref{S2eq1}.
In what follows, we shall always use the convention that $\vec{f}=\left(f_1,\cdots,f_n\right)\in X$ means that
each component, $f_i,$ belongs to the space $X,$ and we designate $\|\vec{f}\|_X\eqdefa \sum_{i=1}^n\|f_i\|_X.$

Given  $\vec{\zeta}=\begin{pmatrix}
\zeta_0\\
\zeta_1
\end{pmatrix}$, $\vec{\xi}=\begin{pmatrix}
\xi_0\\
\xi_1
\end{pmatrix}$  and $\vec{\eta}=\begin{pmatrix}
\eta_0\\
\eta_1
\end{pmatrix},$     we  consider
\begin{equation}\label{S-eq-full-closed}
\left\{\begin{array}{l}
\displaystyle  \pa_t\r_0+\pa_x\bigl((u_0+\zeta_0)\r_0\bigr)=0,\\
\displaystyle \pa_t u_0+\bigl(u_0+\xi_0)\pa_x\bigl(u_0+\xi_0)+\pa_x\r_0=\eta_0,\\
\displaystyle \pa_tA+\bigl(u_0+\zeta_0\bigr)\pa_xA+2\pa_x\bigl(u_0+\zeta_0\bigr)A=0 \quad\mbox{in}\ \ [0,T]\times\TT,\\
\displaystyle \pa_t\r_1+\pa_x\left(\bigl(u_0+\zeta_0\bigr)\r_1+\bigl(u_1+\zeta_1\bigr)\r_0\right)=\pa_x A,
 \\
 \displaystyle  \pa_tu_1+\bigl(u_0+\xi_0)\pa_x\left(u_1+\xi_1\right)+\left(u_1+\xi_1\right)\pa_x\bigl(u_0+\xi_0)+\pa_x\r_1=\eta_1,\\
  \displaystyle \vec{\r}(0)=(g_0,g_1),\quad A(0)=A_0\andf \vec{u}(0)=(v_0,v_1),
\end{array}\right.
\end{equation}
where $\vec{\r}=(\r_0,\r_1)$ and $\vec{u}=(u_0,u_1).$

Let $U\eqdefa (v_0, v_1, g_0, g_1, A_0, \xi_0, \xi_1, \zeta_0, \zeta_1,  \eta_0, \eta_1)$, and
\beno
\begin{split}
X^k(T)\eqdefa& \{U\in H^k_0\times H^{k-2}_0\times H^k\times H^{k-2}\times H^{k-1}\times L^2_T(H^{k+1}_0)  \times L^2_T(H^{k-1}_0)\\
&\times  L^2_T(H^{k+1}_0)\times L^2_T(H^{k-1}_0)\times  L^2_T(H^{k}_0)\times L^2_T(H^{k-2}_0);\, g_0>0 \text{ in }  \mathbb{T}\};\\
Y^k(T)\eqdefa& \{(u_0, u_1, \rho_0, \rho_1, A)\in C([0,T];H^k_0)\times C([0,T];H^{k-2}_0)\times C([0,T];H^k) \\& \times C([0,T];H^{k-2})\times C([0,T];H^{k-1});\, \r_0>0 \text{ in } [0,\tilde T]\times \mathbb{T} \}.
\end{split}
\eeno
Let  us remark here that, concerning $U$, the natural convention should have been
$(g_0, v_0,  A_0, g_1, $ $ v_1,  \xi_0, \xi_1, \zeta_0, \zeta_1,  \eta_0, \eta_1)$.  In this paper we choose a different order  to simplify some notations for the presentation, for example we will denote $(v_0, v_1)$ by $\vec v$. The same remark holds for the use of $(u_0, u_1, \rho_0, \rho_1, A)$.

The main results state as follows:

\begin{theo}\label{th3.5}
{\sl Let $k \geq 3$. Given $U\in X^k(T)$ the system \eqref{S-eq-full-closed} has a unique solution
 $(\vec{u}, \vec{\r}, A)$ in $Y^k(T_0)$ for some $T_0\in(0,T].$ Furthermore, if for some $U^1\in X^k(T),$ the system \eqref{S-eq-full-closed} has a solution $(\vec{u}^1, \vec{\r}^1, A)\in Y^k(T)$, then there exists $\delta$ and $C>0$ which depend only on $\lVert U^1\lVert_{X^k(T)}$   and the uniform lower bound of $\rho_0$   such that
\begin{itemize}
\item[(i)] if $U^2\in X^k(T)$ satisfies
\begin{equation}\label{U_2-delta}
\lVert U^1-U^2\lVert_{X^k(T)}\leq \delta,
\end{equation}
then \eqref{S-eq-full-closed} has a unique solution $(\vec{u}^2, \vec{\r}^2, A^2)\in Y^k(T)$.
\item[(ii)] We define $\cR(U)$ to be the solution of \eqref{S-eq-full-closed}. If $U^2$ satisfies \eqref{U_2-delta}, then
\begin{equation}
\lVert \mathcal{R}\bigl(U^1\bigr)- \mathcal{R}\bigl(U^2\bigr)\lVert_{Y^{k-1}(T)}\leq C \lVert U^1-U^2\lVert_{X^{k-1}(T)}.   \notag
\end{equation}
\item[(iii)] The operator $\mathcal{R}: X^k(T)\rightarrow Y^k(T)$ is continuous at $U^1$.
\end{itemize}}
\end{theo}

\begin{theo}\label{thm-con}
Let $k\geq 3$.  If $(\xi^n_0, \zeta^n_0)$ and $(\xi^n_1, \zeta^n_1)$ are bounded sequences in $L^2_T(H^{k+3})$ and $L^2_T(H^{k+1})$ respectively, which
satisfy, as $n\rightarrow +\infty$,
\begin{gather}
\int_0^{t} \zeta_0^n(s) \chi_0^n(s) ds\rightarrow 0 \textrm{ in }H^k,  \notag\\
\int_0^{t} \zeta_0^n(s) \tilde  \chi_0^n(s) ds\rightarrow 0 \textrm{ in }H^{k-1},  \notag\\
\int_0^{t} \zeta_1^n(s) \chi_1^n(s) ds\rightarrow 0 \textrm{ in }H^{k-2},   \notag
\end{gather}
for any $t$ and for any uniformly equicontinuous sequences $\chi^n=(\chi^n_0, \tilde \chi^n_0, \chi^n_1):[0, T]\rightarrow H^{k}\times H^{k-1}\times H^{k-2}$.
Let
\begin{gather*}
U^n=(\vec{v}, \vec{g}, A_0, \vec{\xi}^n, 0, \vec{\eta}),\;
V^n=(\vec{v}, \vec{g}, A_0, \vec{\xi}^n, \vec{\zeta}^n, \vec{\eta})\in X^{k+2}(T).
\end{gather*}
Suppose that corresponding to $U^n$,  the equation \eqref{S-eq-full-closed} have solutions in $Y^{k+2}(T)$ being uniformly bounded in $Y^{k}(T)$. Then, for $n$ sufficiently large  the equation \eqref{S-eq-full-closed} corresponding to $V^n$ have solutions  which is also uniformly bounded in $Y^{k+2}(T)$ and, as $n\rightarrow +\infty$,
\begin{equation*}
\mathcal{R}\bigl(U^n\bigr)- \mathcal{R}\bigl(V^n\bigr)\rightarrow 0 \textrm{ in } Y^k(T).
\end{equation*}
The same property holds if instead of $U^n$ we assume the existence of solutions  for $V^n$.
\end{theo}

The rest part of this section is devoted to the proofs of the two preceding theorems.
\begin{proof}[Proof of Theorems \ref{th3.5} and \ref{thm-con}] We split the proof of Theorems \ref{th3.5} and \ref{thm-con} into the following steps:

\noindent$\bullet$ {\bf Step 1.} Given $(\zeta_0,\xi_0)\in L^2_T(H^{k+1}_0)$ and $f\in L^2_T(H^k_0),$ we consider
\begin{equation}\label{S-eq-full-1}
\left\{\begin{array}{l}
\displaystyle  \pa_t\r_0+\pa_x\bigl((u_0+\zeta_0)\r_0\bigr)=0,\\
\displaystyle \pa_t u_0+\bigl(u_0+\xi_0)\pa_x\bigl(u_0+\xi_0)+\pa_x\r_0=f,\\
\displaystyle \r_0(0)= g_0, u_0(0)= v_0.
\end{array}\right.
\end{equation}
 Let $k\in \mathbb{N}\setminus \{0,1\}$ and let $0<\tilde T\leq T$,
\beno
\begin{split}
Y_0^k(\tilde T)\eqdefa &\{(\r_0, u_0)\in C([0,\tilde T]; H^k)\times C([0,\tilde T]; H^k); \, \r_0>0 \text{ in } [0,\tilde T]\times \mathbb{T}\};\\
X_0^k(T)\eqdefa &\{U_0\in H^k\times H^k\times L^2_T(H^{k+1})\times L^2_T(H^{k+1})\times L^2_T(H^{k}); \, g_0>0 \text{ in }  \mathbb{T}\}.
\end{split}
\eeno
Let $U_0\eqdefa (g_0,v_0, \zeta_0,\xi_0, f)$. We define
the solution operator
\beq \label{S3eq1} \cR_0: X_0^k(T) \mapsto Y_0^k(\tilde T) \quad\mbox{so that}\quad
\cR_0(U_0)\eqdefa (\r_0, u_0).   \notag
\eeq

The system \eqref{S-eq-full-1} is a  symmetric hyperbolic systems with symmetrizer
$
\begin{pmatrix}
1 & 0 \\
0 & \r_0
\end{pmatrix}$. Standard hyperbolic theory (\cite{1993-Beirao-da-Veiga-CPAM}) ensures that, for every bounded set $B$ of $X_0^k(T)$ such that, for some $c>0$, $\rho_0\geq c$ in $ \mathbb{T}$ for every $U_0\in B$, there exists $\tilde T \in (0,T]$ such that  the operator $\cR_0$  is well-defined from $B$ into
$Y_0^k(\tilde T)$. Furthermore, we have

\begin{theo}[Theorem 2.2 of \cite{2011-Nersisyan-CPDE}]\label{th3.1}
{\sl Let $k\geq 3$. Let $U_0^1\in X_0^k(T).$ We suppose that the system \eqref{S-eq-full-1} has a solution $(\r_0^1, u_0^1)\in Y^k_0(T)$. Then there exists $\delta_0$ and $C>0$ which depend only on $\lVert U_0^1\lVert_{X_0^k}$ and the uniform lower bound of $\rho_0^1$  such that,
\begin{itemize}
\item[(i)] if $U_0^2\in X_0^k(T)$ satisfies
\begin{equation}
\lVert U_0^1-U_0^2\lVert_{X^k_0(T)}\leq \delta_0,   \notag
\end{equation}
then  \eqref{S-eq-full-1} has a unique solution $(\r_0^2, u_0^2)\in Y^k_0(T)$.
\item[(ii)]
\begin{equation}
\lVert \mathcal{R}_0\bigl(U_0^1\bigr)- \mathcal{R}_0\bigl(U_0^2\bigr)\lVert_{Y^{k-1}_0(T)}\leq C \lVert U_0^1-U_0^2\lVert_{X^{k-1}_0(T)}.    \notag
\end{equation}
\item[(iii)] The operator $\mathcal{R}_0: X^k_0(T)\rightarrow Y^k_0(T)$ is well defined and continuous on a neighborhood of $U_0^1$.
\end{itemize}}
\end{theo}

Another important property of $\mathcal{R}_0$ is the following oscillation type lemma.
\begin{theo}[Theorem 2.3 of \cite{2011-Nersisyan-CPDE}]\label{th3.2}
{\sl Let $k\geq 3$. Let $\xi_0^n, \zeta_0^n$ be bounded sequences in $L^2_T(H^{k+3})$ and $\zeta_0^n$ satisfy
\begin{equation}
\int_0^{t} \zeta_0^n(s) \chi_0^n(s) ds\rightarrow 0 \textrm{ in }H^k   \textrm{ as $n\rightarrow +\infty$, }  \notag
\end{equation}
for any $t\in [0, T]$ and for any uniformly equicontinuous sequence $\chi_0^n: [0, T]\rightarrow H^k$.
Let
\begin{gather}
U_0^n=(\r_0, v_0, \xi_0^n, \zeta_0^n, \eta_0)\in X_0^{k+2}(T),  \notag\\
V_0^n=(\r_0, v_0, \xi_0^n, 0, \eta_0)\in X_0^{k+2}(T). \notag
\end{gather}
Suppose that corresponding to $U_0^n$,  the equation \eqref{S-eq-full-1} have solutions in $Y^{k+2}_0(T)$ being uniformly bounded in $Y^{k}_0(T)$. Then, for $n$ sufficiently large  the equation \eqref{S-eq-full-1} corresponding $V^n_0$ have solutions  in $Y^{k+2}_0(T)$, moreover,  there holds
\begin{equation}
\mathcal{R}_0\bigl(U_0^n\bigr)-\mathcal{R}_0\bigl(V_0^n\bigr)\rightarrow 0 \textrm{ in } Y_0^k(T) \textrm{ as $n\rightarrow +\infty$. }   \notag
\end{equation}}
\end{theo}

\noindent$\bullet$ {\bf Step 2.}
 We investigate the linear transport equation
\begin{equation}\label{S-eq-full-2-gen}
\left\{\begin{array}{l}
\displaystyle
 \pa_tA+\bigl(u_A+\zeta_A\bigr)\pa_xA+2\pa_x\bigl(u_A+\zeta_A\bigr)A=0,\\
 \displaystyle
 A(0)= A_0.
 \end{array}\right.
\end{equation}
Let $U_A\eqdefa (A_0, u_A, \zeta_A),$ and
\begin{gather*}
X_A^k(T)\eqdefa H^{k-1}\times L^2_T(H^k)\times L^2_T(H^k) \andf
Y_A^k(T)= C([0,T]; H^{k-1}).
\end{gather*}
We define the solution operator
\begin{equation}
\cR_{\mathcal{A}}: X_A^k(T)\rightarrow Y_A^k(T)\quad \with \cR_{\mathcal{A}}(U_A)=A.   \notag
\end{equation}

\begin{theo}\label{th-transport}
{\sl Let $l\geq 3$. For any $U_A\in X_A^l(T),$
 \eqref{S-eq-full-2-gen} has a unique solution $A$ in $Y_A^l(T)$. Moreover, for every bounded subset $B$ of $X_A^l(T)$ there exists $C>0$ so that
\begin{equation}\label{S3eq11}
\lVert \cR_{\mathcal{A}} \bigl(U_A^1\bigr)- \cR_{\mathcal{A}} \bigl(U_A^2\bigr)\lVert_{Y^{l-1}_A(T)}\leq C\lVert U_A^1-U_A^2\lVert_{X^{l-1}_A}(T), \, \forall U_A^1\in B, \, \forall U_A^2\in B .   \notag
\end{equation}

Furthermore, let $k\geq 3$, let $\zeta_A^n$ be a bounded sequences in $L^2_T(H^{k+2})$  such that
\begin{equation}\label{S3eq12}
\int_0^{t} \zeta_A^n(s) \chi_A^n(s) ds\rightarrow 0 \textrm{ in }H^{k-1}  \textrm{ as } n\rightarrow +\infty,
\end{equation}
for any $t\in [0, T]$ and for any uniformly equicontinuous sequence $\chi_A^n: [0, T]\rightarrow H^{k-1}$.
Let $
U_A^n=(A_0, u_A, \zeta_A^n)\in X_A^{k+2}(T)$,
we denote $V_A=(A_0, u_A, 0).$ Then, we have
\begin{equation}\label{S3eq13}
\cR_{\mathcal{A}}\bigl(U_A^n\bigr)-\cR_{\mathcal{A}}\bigl(V_A\bigr)\rightarrow 0 \textrm{ in } Y_A^k(T).
\end{equation}
}
\end{theo}

\begin{proof}[Proof of Theorem \ref{th-transport}] The first part of Theorem \ref{th-transport} follows  from the explicit value of $\cR_{\mathcal{A}}$ which can be obtained by using the characteristics method.

Let us present the proof of \eqref{S3eq13}. Let us denote
$A^n\eqdefa \cR_{\mathcal{A}} \bigl(U_A^n\bigr), $ $ A\eqdefa \cR_{\mathcal{A}} \bigl(V_A\bigr)$ and $\d_A^n\eqdefa A^n-A$.  Since $\zeta_A^n$ is a bounded sequences in $L^2_T(H^{k+2})$, and
$U_A^n$ is uniformly bounded in $X_A^{k+2}(T)$, we deduce from standard theory of transport equation that \eqref{S-eq-full-2-gen} has a unique solution
$A^n$ which is uniformly bounded in $C([0,T]; H^{k+1}).$  The equation \eqref{S-eq-full-2-gen} implies that $\pa_tA^n$ which is uniformly bounded in $C([0,T]; H^{k}).$ This together with Aubin-Lions Lemma ensures that
\beq \label{S3eq14}
A^n\quad\mbox{is equicontinuous in }\ \ C([0,T]; H^s)\ \forall\ s<k+1.
\eeq
 We know that there exists $f_A\in C([0,T]; H^k)$ such that  as $n\rightarrow +\infty$,
\begin{equation*}
\d_A^n \rightarrow f_A \textrm{ uniformly in } C([0,T]; H^k).
\end{equation*}
 Then one has
\beno
\pa_t\d_A^n+u_A\pa_x\d_A^n+2\pa_xu_A\d_A^n+\zeta_A^n\pa_xA^n+2\pa_x\zeta_A^nA^n=0.
\eeno
By performing $H^{k-1}$ energy estimate to the above equation, we obtain
\begin{equation} \label{S3eq15}
\begin{split}
\frac12\frac{d}{dt}\|\d_A^n(t)\|_{H^{k-1}}^2=&-\left(u_A\pa_x\d_A^n, \d_A^n\right)_{H^{k-1}} -2\left(\pa_xu_A\d_A^n, \d_A^n\right)_{H^{k-1}}\\
&-\left(\zeta_A^n\pa_xA^n, \d_A^n\right)_{H^{k-1}}-2\left(\pa_x\zeta_A^nA^n, \d_A^n\right)_{H^{k-1}}.
\end{split}
\end{equation}
By applying Moser type inequality, we find
\beno\begin{split}
\bigl|\left(u_A\pa_x\d_A^n, \d_A^n\right)_{H^{k-1}}\bigr|\leq &\sum_{\ell\leq k-1}\Bigl(\bigl|\bigl(u_A\pa_x^{\ell+1}\d_A^n, \pa_x^\ell\d_A^n\bigr)_{L^2}\bigr|\\
&\qquad+\bigl|\bigl(\pa_x^\ell(u_A\pa_x\d_A^n)-u_A\pa_x^{\ell+1}\d_A^n, \pa_x^\ell\d_A^n\bigr)_{L^2}\bigr|\Bigr)\\
\leq &C\bigl(\|\pa_x u_A\|_{L^\infty}\|\d_A^n\|_{H^{k-1}}+\|\pa_x\d_A^n\|_{L^\infty}\|u_A\|_{H^{k-1}}\bigr)\|\d_A^n\|_{H^{k-1}}\\
\leq &C\|u_A\|_{H^{k-1}}\|\d_A^n\|_{H^{k-1}}^2,
\end{split}
\eeno
where in the last inequality, we used an integration by parts the assumption that $k\geq 3$ so that $H^{k-2}\hookrightarrow L^\infty.$

Since $H^{k-1}$ is an algebra (note that $k\geq 2$), we have
\beno
\bigl|\left(\pa_xu_A\d_A^n, \d_A^n\right)_{H^{k-1}}\bigr|\leq C\|u_A\|_{H^k}\|\d_A^n\|_{H^{k-1}}^2.
\eeno
Finally, by using integration by parts, we have
\begin{align*}
&-\left(\zeta_A^n\pa_xA^n, \d_A^n\right)_{H^{k-1}}-2\left(\pa_x\zeta_A^nA^n, \d_A^n\right)_{H^{k-1}}\\
=& -\left(\zeta_A^n\pa_xA^n, \d_A^n\right)_{H^{k-1}}+2\left(\zeta_A^n A^n, \pa_x\d_A^n\right)_{H^{k-1}}.
\end{align*}

 We focus on the estimate of  $\left(\zeta_A^n\pa_xA^n, \d_A^n\right)_{H^{k-1}}$ with the other one being similar. Since $\zeta_A^n\pa_xA^n$ is uniformly bounded in $L^2_T(H^{k-1})$ and that $\d_A^n$ converges uniformly to  $f_A$ in $C([0, T]; H^k)$ sense, we have that  for any $t\in [0, T]$, as $n\rightarrow +\infty$,
\begin{equation}
\int_0^t \left(\zeta_A^n\pa_xA^n, \d_A^n\right)_{H^{k-1}} dt- \int_0^t \left(\zeta_A^n\pa_xA^n, f_A\right)_{H^{k-1}} dt\rightarrow 0.
\end{equation}
For any $\delta>0$ we can find a step function $f_m$ in $H^k$ such that $\|f_m(s)-f_A(s)\|_{H^k}\leq \delta$, $\forall s\in [0, T]$. Hence, as $n\rightarrow +\infty$,
\begin{equation}
|\int_0^t \left(\zeta_A^n\pa_xA^n, f_A\right)_{H^{k-1}} dt- \int_0^t \left(\zeta_A^n\pa_xA^n, f_m\right)_{H^{k-1}} dt|\leq C \delta
\end{equation}
with $C$ independent of $\delta$ and $n$. As $f_m(t)$ is a piecewise constant function in $H^k$, for instance assume that the discontinuous points are $\{t_1, t_2,..., t_M\}$, we can immediately conclude from  the assumption, Equation \eqref{S3eq12},  that
\begin{equation*}
 \int_0^{t_1} \left(\zeta_A^n\pa_xA^n, f_m\right)_{H^{k-1}} dt=   \left(\int_0^{t_1}\zeta_A^n\pa_xA^n dt, f_m(0)\right)_{H^{k-1}}  \rightarrow 0,
\end{equation*}
as $n\rightarrow +\infty$,  which further  implies that when $n$ tends to $+\infty$,
\begin{equation*}
 \int_0^{t} \left(\zeta_A^n\pa_xA^n, f_m\right)_{H^{k-1}} dt \rightarrow 0,
\end{equation*}
thus
\begin{equation}
\int_0^t \left(\zeta_A^n\pa_xA^n, \d_A^n\right)_{H^{k-1}} dt\rightarrow 0.
\end{equation}

Therefore,
\begin{equation}\label{eq-321}
\frac12\frac{d}{dt}\|\d_A^n(t)\|_{H^{k-1}}^2\leq  C\|u_A\|_{H^k}\|\d_A^n\|_{H^{k-1}}^2+ h^n(t),
\end{equation}
with $h^n(t)\geq 0$ satisfying
\begin{equation}\label{eq-313}
\int_0^T h^n(s) ds\rightarrow 0.
\end{equation}

Now let us present the following lemma,
\begin{lemma}\label{S2lem1}
{\sl Let $0\leq f, g \in C[0,T]$ so that
\beq\label{S2eq01}
\frac{d}{dt}f(t)\leq g(t)f(t)+h(t)\quad \forall\ t\in [0,T].
\eeq
 Then one has
\beq \label{S2eq02}
f(t)\leq f(0)e^{\int_0^tg(t')\,dt'}+\int_0^th(t')\,dt'+\int_0^tg(t')e^{\int_{t'}^tg(s)\,ds}\int_0^{t'}h(s)\,ds\,dt'.
\eeq}
\end{lemma}
By inserting equations \eqref{eq-321}--\eqref{eq-313} into Lemma \ref{S2lem1}  we get \eqref{S3eq13}.
\end{proof}

\begin{proof}[Proof of Lemma \ref{S2lem1}] We first get, by integrating \eqref{S2eq01} over $[0,t],$ that
\beq\label{S2eq03}
f(t)\leq f(0)+\int_0^tg(t')f(t')\,dt'+\int_0^th(t')\,dt',\eeq
from which, we infer
\beno
\frac{d}{dt}\int_0^tg(t')f(t')\,dt'=g(t)f(t)\leq g(t)\Bigl(f(0)+\int_0^tg(t')f(t')\,dt'+\int_0^th(t')\,dt'\Bigr),
\eeno
so that there holds
\beno
\frac{d}{dt}\Bigl(e^{-\int_0^tg(t')\,dt'}\int_0^tg(t')f(t')\,dt'\Bigr)\leq g(t)e^{-\int_0^tg(t')\,dt'}\Bigl(f(0)+\int_0^th(t')\,dt'\Bigr).
\eeno
Integrating the above inequality over $[0,t]$ gives rise to
\beno
\int_0^tg(t')f(t')\,dt'\leq \int_0^t g(t')e^{\int_{t'}^tg(s)\,ds}\Bigl(f(0)+\int_0^{t'}h(s)\,ds\Bigr)\,dt'
\eeno
Inserting the above inequality into \eqref{S2eq03}
and using the fact that
\beno  \int_0^t g(t') e^{\int_{t'}^t g(s)\,ds}\,dt'=e^{\int_{0}^t g(t')\,dt'}- 1\eeno
leads to \eqref{S2eq02}.
\end{proof}
\\

\noindent$\bullet$ {\bf Step 3.}
We investigate the linear hyperbolic system:
\begin{equation}\label{S-eq-full-3-recall}
\left\{\begin{array}{l}
\displaystyle \pa_t\r_1+\pa_x\left(\bigl(u_0+\zeta_0\bigr)\r_1+\bigl(u_1+\zeta_1\bigr)\r_0\right)=f,
 \\
 \displaystyle  \pa_tu_1+\bigl(u_0+\xi_0)\pa_x\left(u_1+\xi_1\right)+\left(u_1+\xi_1\right)\pa_x\bigl(u_0+\xi_0)+\pa_x\r_1=\eta_1,\\
\displaystyle \r_1(0)= g_1, u_1(0)= v_1.
\end{array}\right.
\end{equation}
 Let us denote
\beno
U_1\eqdefa (g_1, v_1,  \xi_1, \zeta_1, f, \eta_1, \r_0, u_0,  \xi_0, \zeta_0) \andf \cR_1\bigl(U_1\bigr)\eqdefa (\r_1, u_1),
\eeno
and
\beno
\begin{split}
X_1^k(T)\eqdefa & \{ U_1\in H^{k-2}\times H^{k-2}_0\times L^2_T(H^{k-1}_0)\times L^2_T(H^{k-1}_0)\times L^2_T(H^{k-2}_0) \times L^2_T(H^{k-2}_0)\\
&\times L^2_T(H^{k-1})\times L^2_T(H^{k-1}_0) \times L^2_T(H^{k-1}_0)\times L^2_T(H^{k-1}_0); \; \rho_0>0 \textrm{ in } \mathbb{T}\},\\
Y_1^k(T)\eqdefa & C([0,T];H^{k-2})\times C([0,T];H^{k-2}_0).
\end{split} \eeno

Due to \eqref{S-eq-full-3-recall} is a linear symmetric hyperbolic system, similar well-posedness theorem as Theorem \ref{th3.1} holds
for $\cR_1\bigl(U_1\bigr).$ Moreover, there holds
\begin{theo}\label{th3.4}
{\sl Let $\zeta_0^n, \zeta_1^n$ be bounded sequences in $L^2_T(H^{k+1}_0)$ and $\zeta_0^n, \zeta_1^n$ satisfy
\begin{equation*}
\int_0^{t} \zeta_0^n(s) \chi_0^n(s) ds \andf \int_0^{t} \zeta_1^n(s) \chi_0^n(s) ds\rightarrow 0 \textrm{ in }H^{k-2}
\end{equation*}
for any $t\in [0, T]$ and for any uniformly equicontinuous sequence $\chi_0^n: [0, T]\rightarrow H^{k-2}$.
Let
\begin{gather*}
U_1^n=(g_1,  v_1, \xi_1, 0, f, \eta_1, \r_0, u_0,  \xi_0, 0)\in X^{k+2}_1(T),\\
V_1^n=(g_1,  v_1, \xi_1, \zeta_1^n, f, \eta_1, \r_0, u_0,  \xi_0, \zeta_0^n)\in X^{k+2}_1(T).
\end{gather*}
Suppose that corresponding to $U_1^n$,  the equation \eqref{S-eq-full-3-recall} have solutions in $Y^{k+2}_1(T)$ being uniformly bounded in $Y^{k}_1(T)$. Then, for $n$ sufficiently large  the equation \eqref{S-eq-full-3-recall} corresponding $V^n_1$ have solutions  in $Y^{k+2}_1(T)$, moreover,   there holds
 Then there holds
\begin{equation}
\mathcal{R}_1\bigl(U_1^n\bigr)-\mathcal{R}_1\bigl(V_1^n\bigr)\rightarrow 0 \textrm{ in } Y_1^k(T)  \textrm{ as } n\rightarrow +\infty.  \notag
\end{equation}}
\end{theo}

By summarizing the above steps, we conclude the proof of Theorems \ref{th3.5} and \ref{thm-con}.
\end{proof}

\medskip

\setcounter{equation}{0}
\section{The control of the limit systems}\label{sec-cont-lim}

The goal of this section is to present the proof of Theorem \ref{thm-con-lim-main}.
 Before going into the detailed  proof, here we heuristically outline  the main ideas.
\begin{itemize}
\item[1)] In  Section \ref{sub-sec-41}, we construct a sequence of  finite dimensional spaces
\begin{equation*}
E_0\subset E_1\subset E_2\subset ...\subset E_n...
\end{equation*}
whose union is dense in $H^s(\mathbb{T})$.
\item[2)] Thanks to the preceding step, for $\forall \varepsilon>0$, for any given initial data $(\vec{v}, \vec{g})\in H^{k+2}_0\times H^{k}\times H^{k+2}_0\times H^{k}$, $A_0\in H^{k+1}$,  and final state $(\wh{\vec{v}}, \wh{\vec{g}})\in H^{k+2}_0\times H^{k}\times H^{k+2}_0\times H^{k}$ satisfying the conservation of mass, there exists
$E_N-$valued control $\vec{\eta}^{(N)}$ such that the solution
\begin{equation*}
(\vec{u}^{(N)}, \vec{\r}^{(N)}, A^{(N)})\eqdefa \cR \bigl(\vec{v}, \vec{g}, A_0, 0, 0, \vec{\eta}^{(N)}\bigr)\in Y^{k+2}(T),
\end{equation*}
verifies that in the  $ H^{k}_0\times H^{k-2}\times H^{k}_0\times H^{k-2}$ space
\begin{equation*}
\|(\vec{u}^{(N)}(T), \vec{\r}^{(N)}(T))- (\wh{\vec{v}}, \wh{\vec{g}})\|<\varepsilon.
\end{equation*}
Section \ref{sub-sec-42}, more precisely, Theorem \ref{thm-cons-EN}, is devoted to this step.

\item[3)] Next, as the most important step of the proof, in Section \ref{sec:reduc-dim} we show that there exists  $E_{N-1}-$valued control $\vec{\eta}^{(N-1)}$ such that the solution
\begin{gather*}
(\vec{u}^{(N-1)}, \vec{\r}^{(N-1)}, A^{(N-1)})\eqdefa \cR \bigl(\vec{v}, \vec{g}, A_0, 0, 0, \vec{\eta}^{(N-1)}\bigr)\in Y^{k+2}(T),\\
\|(\vec{u}^{(N)}(T), \vec{\r}^{(N)}(T))- (\vec{u}^{(N-1)}(T), \vec{\r}^{(N-1)}(T))\|<\varepsilon,\\
\| \vec{\r}^{(N)}- \vec{\r}^{(N-1)}\|_{C([0, T]; H^k\times H^{k-2})}<\varepsilon,
\end{gather*}
then we can further find $E_{N-2}-$valued control $\vec{\eta}^{(N-2)}$... $E_{0}-$valued control $\vec{\eta}^{(0)}$ to approximate the first constructed trajectory $(\vec{u}^{(N)}, \vec{\r}^{(N)}, A^{(N)})$.   In the terminology of Agrachev--Sarychev method, this   procedure is called ``extension" or  ``saturation", see for instance \cite[Section 5.2]{Agrachev-Sarychev-2005-JMFM}.  \\More precisely, this is again divided into several steps. At first, in Section \ref{subsec-431}--\ref{subsec-433}, we find $\vec{\eta}, \vec \mu_{n} \in E^{N-1}$   such that
\begin{gather*}
 \bigl(\vec{v}, \vec{g}, A_0, \vec \mu_{n}, 0, \bar{\vec{\eta}}\bigr) \textrm{ is uniformly bounded in } X^{k+2}(T), \\
\mathcal{R} \bigl(\vec{v}, \vec{g}, A_0, \vec \mu_{n}, 0, \bar{\vec{\eta}}\bigr) \textrm{ is uniformly bounded in } Y^{k+2}(T), \\
\|\mathcal{R} \bigl(\vec{v}, \vec{g}, A_0, \vec \mu_{n}, 0, \bar{\vec{\eta}}\bigr)- \cR \bigl(\vec{v}, \vec{g}, A_0, 0, 0, \vec{\eta}^{(N)}\bigr)\|_{Y^{k}(T)}\rightarrow 0\text{ as $n\rightarrow+\infty$}.
\end{gather*}
Then in Section \ref{subsec-434} we show that thanks to Theorem \ref{thm-con},
\begin{gather*}
\mathcal{R} \bigl(\vec{v}, \vec{g}, A_0, \vec \mu_{n}, \vec \mu_{n}, \bar{\vec{\eta}}\bigr) \textrm{ is uniformly bounded in } Y^{k+2}(T), \\
\|\mathcal{R} \bigl(\vec{v}, \vec{g}, A_0, \vec \mu_{n}, \vec \mu_{n}, \bar{\vec{\eta}}\bigr)-\mathcal{R} \bigl(\vec{v}, \vec{g}, A_0, \vec \mu_{n}, \vec \mu_{n}, \bar{\vec{\eta}}\bigr)\|_{Y^{k}(T)}\rightarrow 0
\text{ as $n\rightarrow+\infty$}.
\end{gather*}
Next, in Section \ref{subsec-435} we approximate $\vec \mu_{n}$ by smooth (in time) functions  $\vec \mu_{n}^m$. Finally, in Section  \ref{subsec-436} we construct the required $\vec{\eta}^{(N-1)}$.

 \item[4)]
Let us emphasize  here that during the whole precess $\vec \rho$ is close to $\vec \rho^{(N)}$ in $H^{k}$, thus in $C(\mathbb{T})$ and being uniformly away from 0, while $\vec u$ is not necessarily always close to $\vec u^{(N)}$ though $\vec u(T)$  is close to $\vec u^{(N)}(T)$ in $H^k$ sense.

\end{itemize}

\subsection{Introduce  the controlling space}\label{sub-sec-41}
We define
\beq\label{S4eq9}
\begin{split}
E_0\eqdefa & E=  \textrm{span} \{\sin x, \cos x\},\\
E_n\eqdefa &\textrm{span} \bigl\{\sin x, \cos x, \sin 2x, \cos 2x,..., \sin (n+1)x, \cos (n+1)x\ \bigr\},\\
\mathbf{E}_n\eqdefa & E_n\times E_n.
\end{split} \eeq
It is obvious that $E_n\subset C^{\infty}(\mathbb{T})$ and
\begin{lemma}
$\mathbf{E}_{\infty}\eqdefa \bigcup\limits_{k=0}^{\infty} \mathbf{E}_k$ is dense in $H^{s}_0(\mathbb{T})\times H^{s}_0(\mathbb{T}), \forall s>0$.
\end{lemma}
Moreover, one has the following ``Lie bracket" type argument.
\begin{lemma}\label{lem-E}
{\sl Let $n\in \mathbb{N}$. For any $\psi \in E_{n+1}$, there exists $\f, \varphi^1,... \varphi^p \in E_n$ such that
\begin{gather}\label{S4eq10}
\psi=\f- \sum_{i=1}^{p} \varphi^i \pa_x \varphi^i.
\end{gather}}
\end{lemma}
\begin{proof} Notice that  $E_{n+1}= E_n\oplus \textrm{span} \{ \pm \sin (n+2)x, \pm \cos (n+2)x\}$.
It suffices to prove \eqref{S4eq10} with $\psi= \pm \sin (n+2)x$ or $\pm \cos (n+2)x.$  In the case when $n+2= 2m$, we have
\begin{align*}
+ \sin 2mx=& -\frac{2}{m} \cos  mx \pa_x \cos mx,\\
- \sin 2mx=& -\frac{2}{m} \sin  mx \pa_x \sin mx,\\
+\cos 2mx=& -\frac{1}{m} (\sin mx- \cos mx)\pa_x (\sin mx- \cos mx),\\
-\cos 2mx=& -\frac{1}{m} (\sin mx+ \cos mx)\pa_x (\sin mx+ \cos mx).
\end{align*}
While when $n+2=2m+1$,  we observe that
\begin{align*}
\bigl(\sin&(m+1)x \pm\sin mx)\pa_x (\sin(m+1)x\pm \sin mx\bigr)\\
=&\frac{m+1}{2}\sin (2m+2)x+ \frac{m}{2} \sin (2m)x\pm(m+1)\sin mx \cos (m+1)x\\
&\quad \pm m \sin (m+1)x \cos mx\\
=&\frac{m+1}{2}\sin (2m+2)x+ \frac{m}{2} \sin (2m)x \pm \frac{2m+1}{2} \sin (2m+1)x \mp\frac{1}{2} \sin x,
\end{align*}
which gives
\begin{align*}
-\frac{2m+1}{2} \sin (2m+1)x= &-(\sin(m+1)x+\sin mx)\pa_x (\sin(m+1)x+\sin mx)\\
&+ \frac{m+1}{2}\sin (2m+2)x+ \frac{m}{2} \sin 2mx- \frac{1}{2}\sin x\\
=&-(\sin(m+1)x+\sin mx)\pa_x (\sin(m+1)x+\sin mx)\\
&-  \cos  (m+1)x \pa_x \cos (m+1)x+ \Big(\frac{m}{2} \sin 2mx- \frac{1}{2}\sin x \Big),
\end{align*}
and
\begin{align*}
+\frac{2m+1}{2} \sin (2m+1)x=&-(\sin(m+1)x-\sin mx)\pa_x (\sin(m+1)x-\sin mx)\\
&+ \frac{m+1}{2}\sin (2m+2)x+ \frac{m}{2} \sin 2mx+ \frac{1}{2}\sin x \\
=&-(\sin(m+1)x-\sin mx)\pa_x (\sin(m+1)x-\sin mx)\\
&-  \cos  (m+1)x \pa_x \cos (m+1)x+ \Big(\frac{m}{2} \sin 2mx+ \frac{1}{2}\sin x \Big).
\end{align*}
Thanks to the fact that $\sin(m+1)x+\sin mx, \cos (m+1)x, \frac{m}{2} \sin 2mx- \frac{1}{2}\sin x \in E_n,$ we obtain \eqref{S4eq10}
for $\psi= \pm \sin (2m+1)x. $

Along the same line, we deduce from the fact
\begin{align*}
\bigl(\sin&(m+1)x\pm\cos mx)\pa_x (\sin(m+1)x\pm \cos mx\bigr)\\
=&\frac{m+1}{2}\sin (2m+2)x- \frac{m}{2} \sin (2m)x\pm(m+1)\cos mx \cos (m+1)x\\
&\quad \mp m \sin (m+1)x \sin mx\\
=&\frac{m+1}{2}\sin (2m+2)x- \frac{m}{2} \sin (2m)x \pm \frac{2m+1}{2} \cos (2m+1)x \pm\frac{1}{2} \cos x,
\end{align*}
that
\begin{align*}
-\frac{2m+1}{2} \cos (2m+1)x=&-(\sin(m+1)x+\cos mx)\pa_x (\sin(m+1)x+\cos mx)\\
&+ \frac{m+1}{2}\sin (2m+2)x- \frac{m}{2} \sin 2mx+ \frac{1}{2}\cos x \\
=&-(\sin(m+1)x+\cos mx)\pa_x (\sin(m+1)x+\cos mx)\\
&-  \cos  (m+1)x \pa_x \cos (m+1)x+ \Big(-\frac{m}{2} \sin 2mx+ \frac{1}{2}\cos x \Big),
\end{align*}
and
\begin{align*}
+\frac{2m+1}{2} \cos (2m+1)x=&-(\sin(m+1)x-\cos mx)\pa_x (\sin(m+1)x-\cos mx)\\
&+ \frac{m+1}{2}\sin (2m+2)x- \frac{m}{2} \sin 2mx- \frac{1}{2}\cos x \\
=&-(\sin(m+1)x-\cos mx)\pa_x (\sin(m+1)x-\cos mx)\\
&-  \cos  (m+1)x \pa_x \cos (m+1)x+ \Big(-\frac{m}{2} \sin 2mx- \frac{1}{2}\cos x \Big).
\end{align*}
This shows that  \eqref{S4eq10} holds
for $\psi= \pm \cos (2m+1)x. $ This ends the proof of the lemma.
\end{proof}

\subsection{Controllability of \eqref{S2eq1} with $(\mathbf{E}_N)-$valued controls}\label{sub-sec-42}

The main result states as follows:

\begin{theo}\label{thm-cons-EN}
{\sl Let $T>0,$  $\vec{v}, \wh{\vec{v}}\in H^{k+1}_0\times H^{k-1}_0,$  $\vec{g},\wh{\vec{g}}\in H^{k+1}\times H^{k-1}$ and
$A_0\in H^k$,  which satisfy
\beq \label{S4eq3}
\int_{\TT}g_0\,dx=\int_{\TT}\wh{g}_0\,dx=\al_0,\quad \int_{\TT}{g}_1\,dx=\int_{\TT}\wh{g}_1\,dx=\al_1 \andf g_0, \wh{g}_0\geq c_0>0.
\eeq
Then for any $\e>0,$ there exists an some $N\in \mathbb{N}^*$, some  $E_N-$valued control $\vec{\eta}^{(N)}$, and some $\hat{A}(T)\in H^{k}$ such that
\begin{equation}\label{S4eq1}
\bigl\|(\wh{\vec{v}}, \wh{\vec{g}}, \hat A(T))-\cR_T\bigl(\vec{v}, \vec{g}, A_0, 0, 0, \vec{\eta}^{(N)}\bigr)\bigr\|_{Z^k}\leq \varepsilon,
\end{equation}
where $Z^k\eqdefa  H^k_0\times H^{k-2}_0\times H^k\times H^{k-2}\times H^{k-1}.$
}
\end{theo}
We remark here that in this theorem we are dealing with the case that  initial and final values belong to $X^{k+1}(T)$ to get approximation in $Z^k$ sense.  Eventually, in order to prove Theorem \ref{thm-con-lim-main} it suffices to let the initial and final states in the more regular space $X^{k+2}(T)$.

In what follows,
we always denote this solution by
\begin{equation}\label{S4eq2}
(\vec{u}^{(N)}, \vec{\r}^{(N)}, A^{(N)})\eqdefa \cR \bigl(\vec{v}, \vec{g}, A_0, 0, 0, \vec{\eta}^{(N)}\bigr).  \notag
\end{equation}

\begin{proof} Let us define
\beq \label{S2eq4}
\vec{\r}= \begin{pmatrix}
\r_0\\
\r_1
\end{pmatrix}\eqdefa  T^{-1}\left(t\wh{\vec{g}}+(T-t)\vec{g}\right)
\andf \vec{u}=\begin{pmatrix}
u_0\\
u_1
\end{pmatrix}\eqdefa T^{-1}\left(t\wh{\bf v}+(T-t){\bf v}\right).
\eeq
Then due to \eqref{S2eq4}, we have $\int_{\TT}\pa_t\r_0\,dx=0$ so that we can solve $\xi_0\in C^\infty([0,T]; H^{k+1}_0(\TT))$ from the equation
\beq \label{S2eq4a}
\pa_x\left(\r_0\xi_0\right)=-\pa_t\r_0-\pa_x\left(\r_0 u_0\right).
\eeq
Indeed, since $1/\r_0$ is a solution of $\pa_x (\r_0 f)=0$, we are able to find a  solution $\xi_0\in H_0^{k+1}(\mathbb{T})$ of \eqref{S2eq4a}. With such $\xi_0,$ we can solve for $A\in C^\infty([0,T]; H^{k}(\TT))$ via
\begin{equation*}
\left\{\begin{array}{l}
\displaystyle \pa_tA+\bigl(u_0+\xi_0\bigr)\pa_xA+2\pa_x\bigl(u_0+\xi_0\bigr)A=0 \quad\mbox{in}\ \ [0,T]\times\TT,\\
\displaystyle  A(0)=A_0.
\end{array}\right.
\end{equation*}
Let us remark here that the function $\hat{A}(T)$ stated in the theorem is exactly the solution $A(T)$.

 Again due to \eqref{S2eq4}, we have $\int_{\TT}\pa_t\r_1\,dx=0$ so that we can solve $\xi_1\in  C^\infty([0,T]; H^{k-1}_0(\TT))$ from the equation
\beno
\pa_x(\xi_1\r_0)=\pa_xA-\left(\pa_t\r_1+\pa_x\left(\bigl(u_0+\xi_0\bigr)\r_1+u_1\r_0\right)\right).
\eeno

We now define
\beq \label{s2eq5}
\begin{split}
\eta_0\eqdefa& \pa_t u_0+\bigl(u_0+\xi_0)\pa_x\bigl(u_0+\xi_0)+\pa_x\r_0\in C^{\infty}([0,T]; H^k_0),\\
\eta_1\eqdefa&\pa_tu_1+\bigl(u_0+\xi_0)\pa_x\left(u_1+\xi_1\right)+\left(u_1+\xi_1\right)\pa_x\bigl(u_0+\xi_0)+\pa_x\r_1\in C^{\infty}([0,T]; H^{k-2}_0).  \notag
\end{split}
\eeq
Let us  take $\vec{\xi}^\d=\begin{pmatrix}
\xi_0^\d\\
\xi_1^\d
\end{pmatrix}\in C^\infty([0,T]; H^{k+1}_0\times H^{k-1}_0)$ so that $\vec{\xi}^\d(0,x)=\vec{\xi}^\d(T,x)=0$ and
\beno
\lim_{\d\to 0}\|\vec{\xi}^\d-\vec{\xi}\|_{L^2_T(H^{k+1}\times H^{k-1})}=0.
\eeno
It is easy to observe that
 \beq\label{delta-xi}
\cR_T\bigl(\vec{v}, \vec{g}, A_0, \vec{\xi}^\d,{\bf \xi}^\d, {\bf \eta}\bigr)=\cR_T\bigl(\vec{v}, \vec{g}, A_0, 0, 0, \vec{\eta}-
\pa_t\vec{\xi}^\d\bigr).  \notag
\eeq
Then in view of Theorem \ref{th3.5}, we deduce that
\beno\label{S2eq6}
\begin{split}
&\lim_{\d\to 0}\bigl\|(\wh{\vec{v}}, \wh{\vec{g}}, {A}(T))-\cR_T\bigl(\vec{v}, \vec{g},A_0,  0,0, \vec{\eta}-\pa_t\vec{\xi}^\d\bigr)\bigr\|_{Z^k}\\
&=\lim_{\d\to 0}\bigl\|\cR_T\bigl(\vec{v}, \vec{g}, A_0, \vec{\xi}, \vec{\xi}, \vec{\eta}\bigr)-\cR_T\bigl(\vec{v}, \vec{g}, A_0, \vec{\xi}^\d,{\bf \xi}^\d, {\bf \eta}\bigr)\bigr\|_{Z^k}=0.
\end{split}
\eeno
On the other hand,  due to $\mathbf{E}_\infty$  is dense in $H^{k}_0\times H^{k-2}_0,$ we have
\beno
\lim_{N\rightarrow +\infty} \bigl\|\PP_{\mathbf{E}_N}\bigl(\vec{\eta}-\pa_t\vec{\xi}^\d\bigr)-\bigl(\vec{\eta}-\pa_t\vec{\xi}^\d\bigr)\bigr\|_{L^2_T(H^k  \times H^{k-2})}=0,
\eeno
which together with Theorem \ref{th3.5} ensures \eqref{S4eq1} with $\vec{\eta}^{(N)}\eqdefa \PP_{\mathbf{E}_N}\bigl(\vec{\eta}-\pa_t\vec{\xi}^\d\bigr).$
This completes the proof of the theorem.
\end{proof}

\subsection{Reduction of dimension of control space}
\label{sec:reduc-dim}
The aim of this section is to look for  $\vec \eta^{(N-1)}\in \mathbf{E}_{N-1}$ so that
\begin{equation}\label{S4eq8}
\bigl\|(\wh{\vec{v}}, \wh{\vec{g}}, \hat{A}(T))-\cR_T\bigl(\vec{v}, \vec{g}, A_0, 0, 0, \vec{\eta}^{(N-1)}\bigr)\bigr\|_{Z^k}\leq \varepsilon.  \notag
\end{equation}

Let us start by stating the following  lemma which is also called  ``convexification" by Agrachev--Sarychev  \cite[Section 5.3]{Agrachev-Sarychev-2005-JMFM}.  The main novelty for this lemma is on the treatment of $(\zeta_0, \zeta_1)$ involving both nonlinear terms $\xi^i_0 \pa_x \xi^i_0$ and   bilinear terms $\pa_x (\xi^i_0 \xi^i_1)$.
\begin{lemma}\label{lem-fun}
{\sl Let $n\in \mathbb{N}$. For any $\vec{\zeta}= \begin{pmatrix}
\zeta_0\\
\zeta_1
\end{pmatrix} \in \mathbf{E}_{n+1}$, there exists $\vec{\eta}, \vec{\xi}^1,..., \vec{\xi}^m\in \mathbf{E}_n$ such that
\beq\label{S4eq11}
\begin{split}
{\zeta}_0=\eta_0- \sum_{i=1}^{m} \xi^i_0 \pa_x \xi^i_0 \andf
{\zeta}_1=\eta_1- \sum_{i=1}^{m}  \pa_x (\xi^i_0 \xi^i_1).
\end{split} \eeq}
\end{lemma}
\begin{proof}  Since we can extend the length ($m$) of the sequence,
we only need to prove the cases when $\bar{\eta}=(\bar{\eta}_0, 0)$ or $\bar{\eta}=(0,\bar{\eta}_1)$,  Furthermore, we can assume that $\bar{\eta}_0, \bar{\eta}_1= \pm \sin (n+2)x$ or $\pm \cos (n+2)x$.

The case when $\bar{\eta}=(\bar{\eta}_0, 0)$ is a direct consequence of Lemma \ref{lem-E}, We can take $\eta_1, \xi_1^i=0$.

 Let us turn to the  case when $\bar{\eta}=(0, \bar{\eta}_1)$ with $\bar{\eta}_1= \pm \sin (n+2)x$ or $\pm \cos (n+2)x$.
  Indeed it follows from  Lemma \ref{lem-E} that there  exist of $\eta_{1+}, \eta_{1-}, \varphi_{1+}^i, \varphi_{1-}^i \in E_n$ such that
\begin{gather*}
+{\zeta}_1= \eta_{1+}- \sum_{i=1}^p \varphi_{1+}^i \pa_x \varphi_{1+}^i,\\
-{\zeta}_1= \eta_{1-}- \sum_{i=1}^p \varphi_{1-}^i \pa_x \varphi_{1-}^i,
\end{gather*}
which implies
\begin{align*}
0&= (\eta_{1+}+ \eta_{1-})- \sum_{i=1}^p \varphi_{1+}^i \pa_x \varphi_{1+}^i- \sum_{i=1}^p \varphi_{1-}^i \pa_x \varphi_{1-}^i,\\
{\zeta}_1&= \eta_{1+}- \sum_{i=1}^p \pa_x \big( \varphi_{1+}^i \cdot \frac{1}{2}\varphi_{1+}^i \big)- \sum_{i=1}^p \pa_x \big( \varphi_{1-}^i \cdot 0 \big).
\end{align*}
Then \eqref{S4eq11} is proved by taking
\begin{align*}
m\eqdefa &2p,\\
\eta_0\eqdefa & \eta_{1+}+ \eta_{1-},\\
\eta_1 \eqdefa & \eta_{1+},\\
\xi_0^i\eqdefa & \varphi_{1+}^i, \textrm{ when } i\in \{1,..., p\},\\
\xi_0^i\eqdefa & \varphi_{1-}^{i-k}, \textrm{ when } i\in \{p+1,..., 2p\},\\
\xi_1^i\eqdefa & \frac{1}{2} \varphi_{1+}^i, \textrm{ when } i\in \{1,..., p\},\\
\xi_1^i \eqdefa & 0, \textrm{ when } i\in \{p+1,..., 2p\}.
\end{align*}
This completes the proof of Lemma \ref{lem-fun}.
\end{proof}

With the above key lemma, we are going to prove the following proposition. In fact this proposition together with Theorem \ref{thm-cons-EN} immediately lead to Theorem \ref{thm-con-lim-main}.
\begin{prop}\label{lem-E1-E}
{\sl Let $T>0$, $m\in \mathbb{N}$, and $(\vec{v}, \vec{g}, A_0)$ belong to $ H^{k+2}_0\times H^{k}_0\times H^{k+2}\times H^{k}\times H^{k+1}$
 and $\vec{\eta}_1\in C^{\infty}([0,T];\mathbf{E}_{m+1}).$ For any $\e>0,$  there exists $\vec{\eta}\in C^{\infty}([0,T];\mathbf{E}_m)$
  such that,
\beq\label{S4qe12}
\lVert \mathcal{R}_{T}\bigl(\vec{v}, \vec{g}, A_0, 0, 0, \vec{\eta}\bigr)    - \mathcal{R}_{T}\bigl(\vec{v}, \vec{g}, A_0, 0, 0, \vec{\eta}_1\bigr) \lVert_{Z^k}\leq \varepsilon.
\eeq
}
\end{prop}

The rest part of this section is devoted to the proof of Proposition \ref{lem-E1-E}.
Since $\vec{\eta}_1\in C^{\infty}([0, T]; \mathbf{E}_1)$ can be approximated by simple functions (piecewise constant with respect to time),
\begin{equation*}
\vec{\eta}_1^m\rightarrow \vec{\eta}_1 \textrm{ in } L_T^\infty(\mathbf{E}_1) \text{ as $m\rightarrow+\infty$,}
\end{equation*}
it suffices to consider simple functions. Furthermore, thanks to a simple iteration argument  and  the continuity of $\mathcal{R}$ (Theorem \ref{th3.5} (iii)), it suffices to consider the case when $\vec{\eta}_1$ is independent of the time.

 From now on we assume that $\vec{\eta}_1(t)\equiv \vec{\eta}_1\in \mathbf{E}_1$.   Let us denote
\begin{equation}\label{def-uroA}
(\vec{u}, \vec{\r}, A)\eqdefa \mathcal{R}\bigl(\vec{v}, \vec{g}, A_0, 0, 0, \vec{\eta}_1\bigr)\in Y^{k+2}(T).
\end{equation}

\subsubsection{\textbf{Well-chosen $\xi^{j}_i$}}\label{subsec-431}
We want to use Lemma \ref{lem-fun} in order to reduce the control of  $\vec{\eta}_1\in \mathbf{E}_1$ by the control of  $\vec{\eta}_0\in \mathbf{E}$.

\noindent $\bullet$ \textbf{ Use of Lemma \ref{lem-fun}}\\
Thanks to Lemma \ref{lem-fun}, there exists $\bar{\vec{\eta}}, \vec{\xi}^{j} \in \mathbf{E}, j\in \{1,..., m\}$ such that
\begin{gather}
(\eta_1)_0=\bar{\eta}_0- \frac{1}{m}\sum_{j=1}^{m} \xi^{j}_0 \pa_x \xi^{j}_0,\label{eta-1-0-bar}\\
(\eta_1)_1=\bar{\eta}_1- \frac{1}{m}\sum_{j=1}^{m}  \pa_x (\xi^{j}_0 \xi^{j}_1).
\end{gather}
However, equation \eqref{eta-1-0-bar}
{\bf does not} mean that
\begin{equation}
u_0\pa_xu_0-(\eta_1)_0= \frac{1}{m}\sum_{j=1}^{m} (u_0+\xi^{j}_0) \pa_x(u_0+ \xi^{j}_0)-\bar{\eta}_0,  \notag
\end{equation}
because of linear terms $\xi^{j}_0 \pa_x u_0$ and $u_0 \pa_x \xi^{j}_0$.

\noindent $\bullet$ \textbf{Adding the adjoint terms}\\
The way to fix such a linear error is to add the opposite quadratic terms, which is  motivated by a simple formula
\begin{equation}
(a+b)^2+(a-b)^2= 2(a^2+b^2).  \notag
\end{equation}
More precisely,  let
\begin{equation}\label{def-xijlam}
 \la\eqdefa \frac1{2m} \textrm{ and }  \xi^{j+m}\eqdefa -\xi^{j} \;  \textrm{ for}  \; j=1, 2, \cdots, m.
\end{equation}
Then it is easy to observe that
\beq\label{S2eq11}
\begin{split}
&u_0 \pa_xu_0-(\eta_1)_0=\sum_{j=1}^{2m}\la \bigl(u_0+\xi^{j}_0\bigr) \pa_x\bigl(u_0+\xi^{j}_0\bigr)-\bar{\eta}_0,\\
&\pa_x(u_0u_1)-(\eta_1)_1=\sum_{j=1}^{2m}\la \pa_x\bigl((u_0+\xi^{j}_0)(u_1+\xi^{j}_1)\bigr)-\bar{\eta}^\ell_1.
\end{split}
\eeq
Hence,  the following two systems  in $ [0,T]\times\TT$   are equivalent:
\begin{equation*}
\left\{\begin{array}{l}
\displaystyle  \pa_t\r_0+\pa_x(u_0\r_0)=0,\\
\displaystyle \pa_t u_0+u_0\pa_xu_0+\pa_x\r_0= (\eta_1)_0,\\
\displaystyle \pa_tA+u_0\pa_xA+2\pa_xu_0A=0,\\
\displaystyle \pa_t\r_1+\pa_x\left(u_0\r_1+u_1\r_0\right)=\pa_xA,
 \\
 \displaystyle  \pa_tu_1+\pa_x\left(u_0u_1\right)+\pa_x\r_1=(\eta_1)_1,
 \\
  \displaystyle {\vec \r}(0)={\vec g},\quad A(0)=A_0\andf {\vec u}(0)={\vec v},
\end{array}\right.
\end{equation*}
and
\begin{equation}\label{S2eq12}
\left\{\begin{array}{l}
\displaystyle  \pa_t\r_0+\pa_x(u_0\r_0)=0,\\
\displaystyle \pa_t u_0+\sum_{j=1}^{2n}
\la \bigl(u_0+\xi^{j}_0\bigr) \pa_x\bigl(u_0+\xi^{j}_0\bigr)+\pa_x\r_0= \bar{\eta}_0,\\
\displaystyle \pa_tA+u_0\pa_xA+2\pa_xu_0A=0,\\
\displaystyle \pa_t\r_1+\pa_x\left(u_0\r_1+u_1\r_0\right)=\pa_xA,
 \\
 \displaystyle  \pa_t u_1+\sum_{j=1}^{2n}
\la \pa_x\left((u_0+\xi^{j}_0)(u_1+\xi^{j}_1)\right)+\pa_x\r_1= \bar{\eta}_1,
 \\
  \displaystyle {\vec \r}(0)={\vec g},\quad A(0)=A_0\andf {\vec u}(0)={\vec v}.
\end{array}\right.
\end{equation}
The preceding equivalent

\subsubsection{\textbf{From stationary sequence $\{\xi^j\}_{j=1}^{2m}$ to curves $\{\mu_n(t)\}_n$}}
Let us define a periodic function $\mu_i(t)$ with period $1$ as
\beq \label{S2eq15}
\begin{split}
\mu_i(t)\eqdefa&\xi_i^{j}\quad \mbox{if}\ \ t\in \Bigl[ \frac{j-1}{2m}, \frac{j}{2m} \Bigr)
\\
 \mu_i(t)\eqdefa&-\mu_i\left(t-\frac12\right)\quad \mbox{if}\ \ t\in (1/2,1]
 \end{split}
\eeq
for $i=0,1,$ and for $j\in \{1,..., m\}$.  Notice that
\begin{gather}
\sum_{j=1}^{2m}  \xi_i^{j}=0 \; \textrm{ for } i=0,1, \label{sum=0}\\
\int_0^1  \mu_i(t) dt=0 \; \textrm{ for } i=0,1.
\end{gather}

For any $n\in \mathbb{N}^*$, we also define a periodic function with period $T/n$ as
\begin{equation}
\mu_{i,n}(t)\eqdefa \mu_i(\frac{nt}T)\in L^{\infty}_T(E) \; \textrm{ for }  i=0,1,
\end{equation}
which is uniformly bounded in $ L^{\infty}_T(E)$.

Then we can rewrite Equation \eqref{S2eq12}  in $ [0,T]\times\TT$ as
\begin{equation}\label{S2eq16}
\left\{\begin{array}{l}
\displaystyle  \pa_t\r_0+\pa_x(u_0\r_0)=0,\\
\displaystyle \pa_t u_0+\left(u_0+\mu_{0,n}\right)\pa_x\left(u_0+\mu_{0,n}\right)
+\pa_x\r_0=\bar{\eta}_0+f_{0,n},\\
\displaystyle \pa_tA+u_0\pa_xA+2\pa_xu_0A=0 \\
\displaystyle \pa_t\r_1+\pa_x\left(u_0\r_1+u_1\r_0\right)=\pa_xA,
 \\
 \displaystyle  \pa_tu_1+\pa_x\left((u_0+\mu_{0,n})(u_1+\mu_{1,n})\right)+\pa_x\r_1=\bar{\eta}_1+f_{1,n},
 \\
  \displaystyle {\vec \r}(0)={\vec g},\quad A(0)=A_0\andf {\vec u}(0)={\vec v},
\end{array}\right.
\end{equation}
where
\beno
\begin{split}
f_{0,n}(t)\eqdefa& \frac12\pa_x\left(u_0+\mu_{0,n}(t)\right)^2-
\frac12\sum_{j=1}^{2m}
\la \pa_x\bigl(u_0+\xi^{j}_0\bigr)^2,\\
f_{1,n}(t)\eqdefa& \pa_x\Big((u_0+\mu_{0,n}(t))(u_1+\mu_{1,n}(t))\Big)-\sum_{j=1}^{2m}
\la \pa_x\left((u_0+\xi^{j}_0)(u_1+\xi^{j}_1)\right).
\end{split}
\eeno
Therefore,
\begin{equation}\label{eq-u-1-1-trans}
(\vec{u}, \vec{\r},A)= \mathcal{R}\bigl(\bv, \bg, A_0, \vec \mu_{n} ,0,  \bar{ \vec \eta}_0+\vec f_{n}\bigr).
\end{equation}
Because $ \vec u_0, \vec \mu_n, \vec \xi^j$ are uniformly bounded in $L^{\infty}_T(H^{k+2}_0\times H^k_0)$, we know that $ \vec f_n$ are uniformly bounded in $L^{\infty}_T(H^{k+1}_0\times H^{k-1}_0)$.

\subsubsection{\textbf{Removing the extra source term $ \vec f_n$}}\label{subsec-433}
We can see from \eqref{S2eq16} and \eqref{eq-u-1-1-trans} that the ``extra" source term  $ \vec f_n$  does not belong to $\mathbf{E}$. In order to remove $ \vec f_n$, a natural idea is to regard it as a  perturbation term, and to use the continuity of the mapping $\mathcal{R}$.  However, we only know the  uniform boundedness of $ \vec f_n$  in $L^{\infty}_T(H^{k+1}_0\times H^{k-1}_0)$, which is not assumed to be small.  In such a case we prove that with the help of a lemma $f_n$ can  be removed. More precisely, we have the following lemma concerning ``relaxation metric" (name according to  \cite[Section 4.1]{Agrachev-Sarychev-2005-JMFM}).

In the next lemma and in the following, for $\psi\in L^1((0,T);\R^l)$, $l\in \mathbb{N}\setminus\{0\}$, we define
$\cK \psi \in W^{1,1}((0,T);\R^l)$ by
\begin{equation}
\label{defcalK}
\cK \psi(t)\eqdefa \int_0^t  \psi (t')\,dt'.
\end{equation}

\begin{lemma}\label{lemmKtend0}
One has
\begin{equation}\label{con-Kf}
\lim_{n\to\infty} \left\|\cK  \vec f_{n}\right\|_{C([0,T]; H^{k+1}_0\times H^{k-1}_0)}=0.
\end{equation}
\end{lemma}
\begin{proof}
We know that $\vec{u} \in H^{k+2}_0\times H^k_0, \vec \mu_n\in L^{\infty}_T(\mathbf{E}), \vec  \xi^k\in \mathbf{E}$.

\noindent$\bullet$ {\bf Step 1.} We only prove the above limit for   $f_{0,n}$, as the same proof holds for $f_{1,n}$.
More generally, we  prove the following lemma.
\begin{lemma}\label{lem-2.5}
Let $ \xi^j_0$ be given. Let $\mu_{0, n}$ be generated by $ \xi^j_0$ following \eqref{S2eq15}. For any $u_0(t)\in L^2_T(H^{k+2}_0)$, we define
\begin{equation*}
\tilde{f}_{0,n}(u_0)\eqdefa \frac12\pa_x\bigl(u_0(t)+\mu_{0,n}(t)\bigr)^2-
\frac12\sum_{j=1}^{2m}
\la \pa_x\bigl(u_0(t)+\xi^{j}_0\bigr)^2.
\end{equation*}
 Then, for any given $u_0(t)\in L^2_T(H^{k+2}_0)$, we have
\begin{equation}\label{ine-k-til-f}
\lim_{n\to+\infty}\left\|\cK \tilde{f}_{0,n}(u_0)\right\|_{C([0,T]; H^{k+1}\times H^{k-1})}=0.
\end{equation}
\end{lemma}
We notice that, by recalling \eqref{def-xijlam}, \eqref{S2eq15} and \eqref{sum=0},
\begin{align*}
&\;\;\;\lVert \mathcal{K}\tilde{f}_{0,n}(u)- \mathcal{K}\tilde{f}_{0,n}(v)\lVert_{C([0, T];H^{k+1})}\\
&\leq \lVert \tilde{f}_{0,n}(u)- \tilde{f}_{0,n}(v)\lVert_{L^{1}_T(H^{k+1})}\\
&\leq \frac{1}{2}\|\bigl(u(t)+\mu_{0,n}(t)\bigr)^2- \bigl(v(t)+\mu_{0,n}(t)\bigr)^2\|_{L^{1}_T(H^{k+2})}  \\
&\;\;\;\;\;\;\;\;\; + \frac{1}{4m}\|\sum_{j=1}^{2m}
\bigl(u(t)+\xi^{j}_0\bigr)^2- \bigl(v(t)+\xi^{j}_0\bigr)^2\|_{L^{1}_T(H^{k+2})}  \\
&= \frac{1}{2}\|\bigl(u(t)+\mu_{0,n}(t)+ v(t)+\mu_{0,n}(t)\bigr) \bigl(u(t)- v(t)\bigr)\|_{L^{1}_T(H^{k+2})}  \\
&\;\;\;\;\;\;\;\;\; + \frac{1}{2}\|
(u(t)+v(t))(u(t)-v(t))\|_{L^{1}_T(H^{k+2})}  \\
&\leq  \bigl(\lVert u\lVert_{L^{2}_T(H^{k+2})}+ \lVert v \lVert_{L^{2}_T(H^{k+2})}+  \lVert \mu_{0, n} \lVert_{L^{2}_T(H^{k+2})}\bigr) \lVert u-v\lVert_{L^{2}_T(H^{k+2})} \\
&\leq  \bigl(\lVert u\lVert_{L^{2}_T(H^{k+2})}+ \lVert v \lVert_{L^{2}_T(H^{k+2})}+  C\bigr) \lVert u-v\lVert_{L^{2}_T(H^{k+2})},
\end{align*}
where $C$ is independent of $n$.
Thus, it suffices to prove \eqref{ine-k-til-f} for  {\bf simple functions} $u_0(t)$.

The reason why we introduce $\tilde{f}_{0, n}$  is that in $f_{0, n}$ the functions  $\xi_0^k$ depend on $u_0(t)$.\\

\noindent$\bullet$ {\bf Step 2.} Since $u_0, \mu_0(t), \xi_0^j$ are chosen from a finite set, we know that $\mathcal{K} \tilde{f}_{0,n} (t)$ is included in a bounded finite dimensional set. Therefore, for all $t$, $\mathcal{K}f_{0, n}(t)$ is relatively compact.
On the other hand, as $\tilde{f}_{0, n}$ is uniformly bounded in $C([0, T];H^{k+1})$, we know that $\{\mathcal{K}\tilde{f}_{0, n}\}$ is equicontinuous on $[0, T]$.
Hence, by Arzel\`a-Ascoli theorem we only need to prove that
\begin{equation}\label{step-ine-2}
\lim_{n\to+\infty} \mathcal{K} \tilde{f}_{0, n}(t)=0, \textrm{ in } H^{k+1}_0,
\end{equation}
for any $t\in [0, T]$.\\

\noindent$\bullet$ {\bf Step 3.} It further suffices to prove \eqref{step-ine-2} for time-independent $u_0$. The general case (when $u_0(t)$ is timely piecewise constant) can be proved by applying the same approach.
Let us therefore suppose that $u_0$ is a constant function in $E$: $u_0(t,x)=\varphi(x)$ for some $\varphi\in E$.  We define the quadratic structure by
\begin{equation*}
B(y)\eqdefa \frac{1}{2} \pa_x y^2.
\end{equation*}
Thus
\begin{equation*}
\tilde{f}_{0,n}(t)= B(u_0+ \mu_{0, n})- \sum_{j=1}^{2m} \lambda B(u_0+ \xi_0^j).
\end{equation*}
The $T/n$-periodicity of $\mu_{0, n}$ tells us that $ (u_0+ \mu_{0, n})$ is periodic with period $T/n$, which to be combined with the fact that $B(u_0+ \xi_0^j)$ is stationary with respect to time,  implies that
$$\tilde{f}_{0, n}(t+ \frac{T}{n})= \tilde{f}_{0, n}(t).$$
Moreover, thanks to the construction of $\mu_{0, n}$, we also have that
\begin{align*}
\int_0^{T/n} \tilde{f}_{0, n} (t) dt&= \int_0^{T/n} \Big(B\bigl(u_0+ \mu_{0, n}(t)\bigr)- \sum_{k=1}^{2m} \lambda B\bigl(u_0+ \xi_0^k\bigr) \Big)dt\\
&= \frac{1}{2}\partial_x\left( \int_0^{T/n}  \bigl(u_0+ \mu_{0, n}(t)\bigr)^2- \lambda \sum_{k=1}^{2m} \bigl(u_0+ \xi_0^k\bigr)^2 dt\right)\\
&= \frac{1}{2}\partial_x\left( \int_0^{T/n}  \bigl(\mu_{0, n}(t)\bigr)^2- \lambda \sum_{k=1}^{2m} \bigl(\xi_0^k\bigr)^2 dt\right)\\
&= \frac{1}{2}\partial_x\left(  \frac{T}{2nm} \sum_{k=1}^{2m} \bigl(\xi_0^k\bigr)^2 -  \lambda \frac{T}{n} \sum_{k=1}^{2m} \bigl(\xi_0^k\bigr)^2 \right)\\
&=0.
\end{align*}

It is known that $\tilde{f}_{0, n}$ is uniformly bounded in $L^{\infty}(H^{k+1})$. Thus, for every $t\in [0,T]$,
\begin{equation*}
\lVert \mathcal{K} \tilde{f}_{0,n}(t)\lVert_{H^{k+1}_0}\leq C \frac{T}{n},
\end{equation*}
which completes the proof of Lemma \ref{lem-2.5}.
\end{proof}\\

Let us define
\begin{equation}\label{w-def}
w_{i,n}\eqdefa u_i-\cK f_{i,n},
\end{equation}
for $i=0,1$. Then \eqref{S2eq16} can be equivalently written as
\begin{equation}\label{S2eq17}
\left\{\begin{array}{l}
\displaystyle  \pa_t\r_0+\pa_x\bigl((w_{0,n}+\cK f_{0,n})\r_0\bigr)=0,\\
\displaystyle \pa_t w_{0,n}+\frac12\pa_x\bigl(w_{0,n}+\mu_{0,n}+\cK f_{0,n}\bigr)^2
+\pa_x\r_0=\bar{\eta}_0,\\
\displaystyle \pa_tA+\bigl(w_{0,n}+\cK f_{0,n}\bigr)\pa_xA+2\pa_x\bigl(w_{0,n}+\cK f_{0,n}\bigr)A=0,\\
\displaystyle \pa_t\r_1+\pa_x\left(\bigl(w_{0,n}+\cK f_{0,n}\bigr)\r_1+\bigl(w_{1,n}+\cK f_{1,n}\bigr)\r_0\right)=\pa_xA,
 \\
 \displaystyle  \pa_tw_{1,n}+\pa_x\bigl((w_{0,n}+\mu_{0,n}+\cK f_{0,n})(w_{1,n}+\mu_{1,n}+\cK f_{1,n})\bigr)+\pa_x\r_1=\bar{\eta}_1,
 \\
  \displaystyle {\vec \r}(0)={\vec g},\quad A(0)=A_0\andf {\vec u}(0)={\vec v},
\end{array}\right.
\end{equation}
which means that
\begin{equation*}
(\vec{w}_n, \vec{\r}, A)=\mathcal{R} \bigl(\vec{v}, \vec{g}, A_0, \vec \mu_{n}+\cK \vec f_{n}, \cK \vec f_{n}, \bar{\vec \eta}\bigr).
\end{equation*}
Equations \eqref{def-uroA}, \eqref{con-Kf} and  \eqref{w-def} tell us that $(\vec{w}_n, \vec{\r}, A)\in Y^{k+1}(T)$.
It further reduces to the following problem:
\begin{equation}\label{S2eq19}
\left\{\begin{array}{l}
\displaystyle  \pa_t\bar{\r}_{0,n}+\pa_x\bigl(\bar{w}_{0,n}\bar{\r}_{0,n}\bigr)=0,\\
\displaystyle \pa_t\bar{w}_{0,n}+\frac12\pa_x\bigl(\bar{w}_{0,n}+\mu_{0,n}\bigr)^2
+\pa_x\bar{\r}_{0,n}=\bar{\eta}_0,\\
\displaystyle \pa_t\bar{A}_n+\bar{w}_{0,n}\pa_x\bar{A}_n+2\pa_x\bar{w}_{0,n}\bar{A}_n=0,\\
\displaystyle \pa_t\bar{\r}_{1,n}+\pa_x\left(\bar{w}_{0,n}\bar{\r}_{1,n}+\bar{w}_{1,n}\bar{\r}_{0,n}\right)=\pa_x\bar{A}_n,
 \\
 \displaystyle  \pa_t\bar{w}_{1,n}+\pa_x\bigl((\bar{w}_{0,n}+\mu_{0,n})(\bar{w}_{1,n}+\mu_{1,n})\bigr)+\pa_x\bar{\r}_{1,n}=\bar{\eta}_1,
 \\
  \displaystyle \bar{{\vec \r}}(0)={\vec g},\quad \bar{A}_n(0)=A_0\andf \bar{{\vec u}}(0)={\vec v},
\end{array}\right.
\end{equation}
which means
\begin{equation}
(\bar{\vec{w}}_n, \bar{\vec{\r}}_n, \bar{A}_n)=\mathcal{R} \bigl(\vec{v}, \vec{g}, A_0, \vec \mu_{n}, 0, \bar{\vec{\eta}}\bigr).           \notag
\end{equation}

We know from Theorem \ref{th3.5},  Lemma \ref{lem-2.5} and  the fact that both  $\lVert (\vec{w}_n, \vec{\r}, A) \lVert_{Y^{k+1}(T)}$ and $\|\bigl(\vec{v}, \vec{g}, A_0, \vec \mu_{n}+\cK \vec f_{n}, \cK \vec f_{n}, \bar{\vec \eta}\bigr)\|_{X^k(T)}$ are uniformly bounded,  that  $\|\mathcal{R} \bigl(\vec{v}, \vec{g}, A_0, \vec \mu_{n}, 0, \bar{\vec{\eta}}\bigr)\|_{Y^k(T)}$ is uniformly bounded for large $n$   and that $\vec \rho_{0, n}(x)$ are uniformly away from 0. Since the norm of $\|\bigl(\vec{v}, \vec{g}, A_0, \vec \mu_{n}, 0, \bar{\vec{\eta}}\bigr)\|_{X^{k+2}(T)}$ is uniformly bounded, thanks to the blow up criteria we know that  $(\bar{\vec{w}}_n, \bar{\vec{\r}}_n, \bar{A}_n)$ are uniformly bounded in $Y^{k+2}(T)$ for large $n$. \\
Because for $n$ large enough  $\lVert (\vec{w}_n, \vec{\r}, A) \lVert_{Y^{k+1}(T)}$,  $\|(\bar{\vec{w}}_n, \bar{\vec{\r}}_n, \bar{A}_n)\|_{Y^{k+1}(T)}, \|\cK  \vec f_{n}\|_{C([0,T]; H^{k+1}_0\times H^{k-1}_0)}$ and  $\|\bigl(\vec{v}, \vec{g}, A_0, \vec \mu_{n}, 0, \bar{\vec{\eta}}\bigr)\|_{X^{k+1}(T)}$ are uniformly bounded, and $\vec \rho_{0, n}$ are uniformly away from 0,  we have  Lipschitz property in lower regularity spaces, namely in $X^k(T)$ and $Y^k(T)$ space,
\begin{equation*}
\lVert (\bar{\vec{w}}_n, \bar{\vec{\r}}_n, \bar{A}_n)- (\vec{w}_n, \vec{\r}, A)
\lVert_{Y^k(T)}\leq C \|\bigl(\vec{v}, \vec{g}, A_0, \vec \mu_{n}+\cK \vec f_{n}, \cK \vec f_{n}, \bar{\vec \eta}\bigr)- \bigl(\vec{v}, \vec{g}, A_0, \vec \mu_{n}, 0, \bar{\vec \eta}\bigr)\|_{X^k(T)},
\end{equation*}
which, together with Lemma \ref{lemmKtend0}, implies that
\begin{equation}\label{ine-w-bar-w}
\lim_{n\rightarrow +\infty} \lVert (\bar{\vec{w}}_n, \bar{\vec{\r}}_n, \bar{A}_n)- (\vec{w}_n, \vec{\r}, A)
\lVert_{Y^k(T)}=0.
\end{equation}
Finally,  \eqref{con-Kf}, \eqref{w-def} and \eqref{ine-w-bar-w} imply that
\begin{equation}
\lim_{n\rightarrow +\infty} \lVert (\bar{\vec{w}}_n, \bar{\vec{\r}}_n, \bar{A}_n)- (\vec{u}, \vec{\r}, A)
\lVert_{Y^k(T)}=0.   \notag
\end{equation}

\subsubsection{\textbf{Continuity property: on the use of Theorem \ref{thm-con}}}\label{subsec-434}
At first we prove the following lemma.
\begin{lemma}
For any $t_0>0$, any uniformly equicontinuous function $\vec \chi_n(t)\in C([0,T]; H^k\times H^{k-2})$, and any uniformly equicontinuous function $\tilde \chi_n(t)\in C([0,T]; H^{k-1})$, we have
\beno
\lim_{n\to +\infty}\bigl\|\int_0^{t_0}\mu_{0, n}(t) \chi_{0, n}(t)\,dt \bigr\|_{H^k}=0,
\eeno
\beno
\lim_{n\to +\infty}\bigl\|\int_0^{t_0} \mu_{0, n}(t)\tilde\chi_{n}(t)\,dt \bigr\|_{H^{k-1}}=0,
\eeno
\beno
\lim_{n\to +\infty}\bigl\|\int_0^{t_0}\mu_{1, n}(t)\chi_{1,n}(t)\,dt \bigr\|_{H^{k-2}}=0.
\eeno
\end{lemma}
\begin{proof}
The proof is straightforward, thanks to the fact that $\mu_n$ behave like an oscillation. Indeed,
\begin{align*}
\int_0^{t_0} \mu_{0,n} \chi_{0, n}(t) dt=&\int_0^{t_0} \mu_0 (\frac{nt}{T})(t) \chi_{0, n}(t) dt\\
=&\frac{T}{n} \int_0^{\frac{n t_0}{T}} \mu_0 (t) \chi_{0, n}(\frac{tT}{n}) dt\\
=& \sum_{i=0}^{[\frac{n t_0}{T}]-1}  \frac{T}{n} \int_i^{i+1} \mu_0 (t) \chi_{0, n}(\frac{tT}{n}) dt+     \frac{T}{n} \int_{[\frac{n t_0}{T}]}^{\frac{n t_0}{T}} \mu_0 (t) \chi_{0, n}(\frac{tT}{n}) dt.
\end{align*}
Since $\vec \chi_n(t)\in C([0,T]; H^k\times H^{k-2})$ is equi-continuous, and $\mu_0(t)$ is in $L^{\infty}_T(\mathbf{E})$ (hence uniformly bounded in $L^{\infty}_T(H^s), \forall s> 0$), we know that
\begin{equation}
\int_{[\frac{n t_0}{T}]}^{\frac{n t_0}{T}} \mu_0 (t) \chi_{0, n}(\frac{tT}{n}) dt    \notag
\end{equation}
is uniformly bounded in $H^k$ for $t_0\in [0, T]$ and for $n$.  On the other hand, we know from the construction of $\mu_0(t)$ that
\begin{align*}
\int_i^{i+1} \mu_0 (t) \chi_{0, n}(\frac{tT}{n}) dt=& \int_i^{i+1/2} \mu_0 (t) \chi_{0, n}(\frac{tT}{n}) dt+  \int_{i+1/2}^{i+1} -\mu_0 (t-\frac{1}{2}) \chi_{0, n}(\frac{tT}{n}) dt\\
=& \int_i^{i+1/2} \mu_0 (t) \Big(\chi_{0, n}(\frac{tT}{n})- \chi_{0, n}(\frac{tT}{n}+ \frac{T}{2n}) \Big)dt.
\end{align*}
Since $\vec \chi_n: [0,T]\rightarrow H^k\times H^{k-2}$ is uniformly equicontinuous, for any $\delta>0$ there exists $M$ such that when $n>M$ we have
\begin{equation}
\lVert \chi_{0, n}(\frac{tT}{n})- \chi_{0, n}(\frac{tT}{n}+ \frac{T}{2n})\lVert_{H^k}<\delta.  \notag
\end{equation}
Hence
\begin{equation}
\lVert \mu_0(t) \Big(\chi_{0, n}(\frac{tT}{n})- \chi_{0, n}(\frac{tT}{n}+ \frac{T}{2n})\Big)\lVert_{H^k}<\delta, \forall 0<t<n.  \notag
\end{equation}
Therefore,
\begin{equation}
 \sum_{i=0}^{[\frac{n t_0}{T}]-1}  \frac{T}{n} \int_i^{i+1} \mu_0 (t) \chi_{0, n}(\frac{tT}{n}) dt< [\frac{n t_0}{T}] \frac{T}{n} \frac{\delta}{2}< T \delta. \notag
\end{equation}
 The same proof  shows that
\beno
\lim_{n\to +\infty}\bigl\|\int_0^{t_0}\mu_{0, n}\tilde \chi_{n}(0, t)\,dt \bigr\|_{H^{k-1}}=0.
\eeno
The same proof also holds for $\mu_1$.
\end{proof}
\\

Following Theorem \ref{thm-con}, we compare   \eqref{S2eq19} to the following equation  in $ [0,T]\times\TT$ ,
\begin{equation}\label{S2eq20}
\left\{\begin{array}{l}
\displaystyle  \pa_t\tilde{\r}_{0,n}+\pa_x\bigl((\tilde{w}_{0,n}+\mu_{0,n})\tilde{\r}_{0,n}\bigr)=0,\\
\displaystyle \pa_t\tilde{w}_{0,n}+\frac12\pa_x\left(\tilde{w}_{0,n}+\mu_{0,n}\right)^2
+\pa_x\tilde{\r}_{0,n}=\bar{\eta}_0,\\
\displaystyle \pa_t\tilde{A}_n+(\tilde{w}_{0,n}+\mu_{0,n})\pa_x\tilde{A}_n+2\pa_x(\tilde{w}_{0,n}+\mu_{0,n})\tilde{A}_n=0\\
\displaystyle \pa_t\tilde{\r}_{1,n}+\pa_x\Big((\tilde{w}_{0,n}+\mu_{0,n})\tilde{\r}_{1,n}+(\tilde{w}_{1,n}+\mu_{1,n})\tilde{\r}_{0,n}\Big)=\pa_x\tilde{A}_n,
 \\
 \displaystyle  \pa_t\tilde{w}_{1,n}+\pa_x\Big((\tilde{w}_{0,n}+\mu_{0,n})(\tilde{w}_{1,n}+\mu_{1,n})\Big)+\pa_x\tilde{\r}_{1,n}
 =\bar{\eta}_1,
 \\
  \displaystyle \tilde{{\vec \r}}_n(0)={\vec g},\quad\tilde{A}_n(0)=A_0\andf \tilde{{\vec u}}_n(0)={\vec v}.
\end{array}\right.
\end{equation}
As $(\bar{\vec{w}}_n, \bar{\vec{\r}}_n, \bar{A}_n)$ are uniformly bounded in $Y^{k+2}(T)$, thanks to Theorem \ref{thm-con} and the blow up criteria,  we know that
\begin{gather}
\lim_{n\rightarrow +\infty} \lVert (\bar{\vec{w}}_n, \bar{\vec{\r}}_n, \bar{A}_n)- (\tilde{\vec{w}}_n, \tilde{\vec{\r}}_n, \tilde{A}_n) \lVert_{Y^k(T)}=0, \notag \\
(\tilde{\vec{w}}_n, \tilde{\vec{\r}}_n, \tilde{A}_n)=\mathcal{R} \bigl(\vec{v}, \vec{g}, A_0, \vec{\mu}_{n}, \vec{\mu}_{n}, \bar{\vec{\eta}}\bigr) \textrm{ is uniformly bounded in } Y^{k+2}(T) \textrm{ for large } n. \notag
\end{gather}
\begin{remark}
We observe that the $L^{\infty}_T(H^s\times H^{s-2})$ norm of $\vec \mu_n$ (with $s>0$) and the source term $\bar{\vec{\eta}}$ only depend on $\vec \eta_1$.   This point probably could  be used to reduce the regularity of the initial state.  However, this is not the main purpose of this paper.
\end{remark}

\subsubsection{\textbf{Approximation of $\mu_n$ by smooth (with respect to time) functions}}\label{subsec-435}

In this step we approximate $\vec \mu_n$ by smooth $\vec \mu_n^m$ such that  their values at the end points are 0.
Let us take $\vec \mu^m_{n} \in C^\infty([0,T]; \mathbf{E})$ with $\vec \mu^m_{n}(0)= \vec\mu^m_{n}(T)=\mathbf{0}$ and
\beno
\lim_{m\to \infty}\|\vec{\mu}^m_{n}-\vec{\mu}_{n}\|_{L^2_T( H^{k+3}_0\times H^{k+1}_0)}=0.
\eeno
Let
\begin{equation}
(\tilde{\vec{w}}_n^m, \tilde{\vec{\r}}^m_n, \tilde{A}^m_n)\eqdefa \mathcal{R} \bigl(\vec{v}, \vec{g}, A_0, \vec{\mu}_{n}^m, \vec{\mu}_{n}^m, \bar{\vec{\eta}}\bigr)\in Y^{k+2}(T).                \notag
\end{equation}
Then, for any $n$,  thanks to the continuity property  of $\mathcal{R}$ stated in Theorem \ref{th3.5}, we have
\begin{equation}
\lim_{ m\rightarrow \infty} \lVert (\tilde{\vec{w}}_n^m, \tilde{\vec{\r}}_n^m, \tilde{A}_n^m)- (\tilde{\vec{w}}_n, \tilde{\vec{\r}}_n, \tilde{A}_n)\lVert_{Y^{k+2}(T)}=0.     \notag
\end{equation}
Consequently, we are able to select a sequence $\{m(n)\}_{\mathbb{N}}$ such that
\begin{equation}
\lim_{ n\rightarrow +\infty} \lVert (\tilde{\vec{w}}_n^{m(n)}, \tilde{\vec{\r}}_n^{m(n)}, \tilde{A}_n^{m(n)})- (\tilde{\vec{w}}_n, \tilde{\vec{\r}}_n, \tilde{A}_n)\lVert_{Y^{k+2}(T)}=0 \textrm{ as $n \rightarrow + \infty$.}     \notag
\end{equation}

\subsubsection{\textbf{Extension technique}}\label{subsec-436}
By taking ${\vec{w}}_{n}^{m(n)}=\tilde{\vec{w}}_{n}+\vec{\mu}_{n}^{m(n)}$, $\vec{\r}_n^{m(n)}= \tilde{\vec{\r}}_n^{m(n)}$ and $A_n^m= \tilde{A}_n^{m(n)}$, we get that in $ [0,T]\times\TT$
\begin{equation}\label{S2eq21}
\left\{\begin{array}{l}
\displaystyle  \pa_t{\r}_{0,n}^{m(n)}+\pa_x\bigl({w}_{0,n}^{m(n)} {\r}_{0,n}^{m(n)}\bigr)=0,\\
\displaystyle \pa_t{w}_{0,n}^{m(n)}+\frac12\pa_x\left({w}_{0,n}^{m(n)}\right)^2
+\pa_x{\r}_{0,n}^{m(n)}=\bar{\eta}_0+ \pa_t\mu_{0,n}^{m(n)},\\
\displaystyle \pa_t{A}_n^{m(n)}+{w}_{0,n}^{m(n)}\pa_xA_n^{m(n)}+2\pa_x{w}^{m(n)}_{0,n}A_n^{m(n)}=0,\\
\displaystyle \pa_t{\r}^{m(n)}_{1,n}+\pa_x\left({w}_{0,n}^{m(n)} {\r}_{1,n}^{m(n)}+{w}^{m(n)}_{1,n}{\r}_{0,n}^{m(n)}\right)=\pa_xA_n^{m(n)},
 \\
 \displaystyle  \pa_t{w}^{m(n)}_{1,n}+\pa_x\left({w}^{m(n)}_{0,n}{w}_{1,n}^{m(n)}\right)+\pa_x{\r}_{1,n}^{m(n)}
 =\bar{\eta}_1+ \pa_t\mu_{0,n}^{m(n)},
 \\
  \displaystyle {{\vec \r}}_n^{m(n)}(0)={\vec g},\quad A_n^{m(n)}(0)=A_0\andf {{\vec u}}_n^{m(n)}(0)={\vec v}.
\end{array}\right.
\end{equation}
Hence, as $n$ tends to $+\infty$,
\begin{equation}
(\vec{w}_n^{m(n)}, \vec{\r}^{m(n)}_n, A^{m(n)}_n)\eqdefa \mathcal{R} \bigl(\vec{v}, \vec{g}, A_0, 0, 0, \bar{\vec{\eta}}+ \pa_t \vec{\mu}^{m(n)}_n\bigr)\in Y^{k+2}(T),     \notag
\end{equation}
\begin{equation}
(\vec{w}_n^{m(n)}, \vec{\r}^{m(n)}_n, A^{m(n)}_n)(T)= (\tilde{\vec{w}}_n^{m(n)}, \tilde{\vec{\r}}^m_n, \tilde{A}^{m(n)}_n)(T)\rightarrow (\vec{u}, \vec{\r}, A)(T) \textrm{ in } Z^k.            \notag
\end{equation}

\section{Semiclassical limit of the controlled system }
In the end, this section is devoted to the proof of Theorem \ref{thm-sem-main}.

\begin{proof}[Proof of Theorem \ref{thm-sem-main}]
The main idea of the proof basically follows from \cite{1998-Grenier-PAMS}, which we outline as the following three steps:

\noindent$\bullet$ {\bf Step 1.}  \textit{There exists $T_0>0$ such that \eqref{S1eq1}--\eqref{cond-reg-be} have solutions which are uniformly bounded for $a^{\hbar}$ in $C([0, T_0]; H^{k})$  and for $S^{\hbar}$ in $C([0, T_0]; H^{k+1})$.}

At first we study the system of $u^{\hbar}$ and $(a^{\rm r, \hbar}, a^{\rm i, \hbar})$:
\begin{equation}\label{sys-u-a}
\left\{\begin{array}{l}
\displaystyle \pa_ta^{\rm r, \hbar}+u^{ \hbar}\pa_xa^{\rm r, \hbar}+\frac12a^{\rm r, \hbar}\pa_xu^{\hbar}=-\frac{\hbar}{2} \pa_x^2 a^{\rm i, \hbar},\\
\displaystyle \pa_ta^{\rm i, \hbar}u^{ \hbar}\pa_xa^{\rm i, \hbar}+\frac12a^{\rm i, \hbar}\pa_x u^{\hbar}= \frac{\hbar}{2} \pa_x^2 a^{\rm r, \hbar},\\
\displaystyle \pa_t u^{\hbar}+u^{\hbar}\pa_x u^{\hbar}+2a^{\rm r, \hbar}\pa_x a^{\rm r, \hbar}+ 2a^{\rm i, \hbar}\pa_x a^{\rm i, \hbar}=-\eta^{\hbar}(t, x),
 \\
\displaystyle (a^{\rm r, \hbar},a^{\rm i, \hbar},u^{\hbar})|_{t=0}=\left({\rm  Re} (a^0(\hbar)), {\rm Im} (a^0(\hbar)),\pa_xS\right).
\end{array}\right.
\end{equation}
This system can be written in the form
\begin{equation}\label{eq-w-im}
\pa_t w^{\hbar}+ A(w^{\hbar}) \pa_x w^{\hbar}= \hbar H(w^{\hbar})+ f^{\hbar},
\end{equation}
\begin{gather}
w^{\hbar}=
\begin{pmatrix}
a^{r, \hbar}  \\
a^{i, \hbar} \\
u^{\hbar}
\end{pmatrix},\;\;\;   A(w^{\hbar})=
\begin{pmatrix}
u^{\hbar} & 0 & \frac{1}{2} a^{\rm r, \hbar}  \\
0 & u^{\hbar} &  \frac{1}{2} a^{\rm i, \hbar} \\
2a^{\rm r, \hbar} & 2a^{\rm i, \hbar} & u^{\hbar}
\end{pmatrix}, \\
H(w^{\hbar})=
\begin{pmatrix}
-\frac{1}{2} \pa_x^2 a^{\rm i, \hbar}  \\
\frac{1}{2} \pa_x^2 a^{\rm r, \hbar} \\
0
\end{pmatrix}, \;\;\;
f^{\hbar}(t, x)= f_0(t, x)+ \hbar f_1(t, x)=
\begin{pmatrix}
0  \\
0 \\
-\eta^{\hbar}(t, x)
\end{pmatrix}.
\end{gather}
We observe that this system shares the following properties.
\begin{itemize}
\item[(i)] The maps $w\mapsto A(w)$ and $w\mapsto H(w)$ are linear.
\item[(ii)]  The left hand side of \eqref{eq-w-im} is symmetrizable with diagonal matrix $A_0:= \rm{diag} (1, 1, 1/4)$.
\item[(iii)]  $H(w^{\hbar})$ has an Hamiltonian structure.
\item[(iv)]  $f^{\hbar}(t, x)$ is uniformly bounded in $C([0, T]; H^{k+4})$.
\end{itemize}
Therefore, for $\alpha\leq k$ classical energy estimates lead to
\beq\label{ene-w}
\begin{split}
\pa_t(A_0 \pa_x^{\alpha} w^{\hbar}, \pa_x^{\alpha} w^{\hbar})
=& 2 \hbar (A_0 H(\pa_x^{\alpha} w^{\hbar}), \pa_x^{\alpha} w^{\hbar})\\
&- 2(A_0 \pa_x^{\alpha} (A(w^{\hbar}) \pa_x w^{\hbar}), \pa_x^{\alpha} w^{\hbar})+ 2(A_0 \pa_x^{\alpha} f^{\hbar}, \pa_x^{\alpha} w^{\hbar}).
\end{split} \eeq
Thanks to the special structure of $H(w^{\hbar})$, an integration by parts gives
\begin{equation}
(A_0 H(\pa_x^{\alpha} w^{\hbar}), \pa_x^{\alpha} w^{\hbar})=0.  \notag
\end{equation}
Moreover, since $\pa_x^{k} f^{\hbar}(t, x)$ are uniformly bounded on $[0, T]\times \mathbb{T}$ and using the fact that $A_0$ is a symmetrizer for \eqref{eq-w-im}, classical estimates  on the right hand side of  \eqref{ene-w} lead to
\begin{equation}
\pa_t \sum_{\alpha \leq k} (A_0 \pa_x^{\alpha} w^{\hbar}, \pa_x^{\alpha} w^{\hbar}) \leq C  \lVert w^{h} \lVert_{C^1}  \lVert w^{h} \lVert_{H^k}^2  + C \lVert w^{h} \lVert_{H^k},   \notag
\end{equation}
which $C>0$ is a constant independent of $\hbar\in (0,1]$.  As the initial data are uniformly bounded in $H^k$ (see \eqref{cond-reg-be}),  by  Gronwall's lemma we get the existence of  $T_0$ such that $a^{\hbar}, u^{\hbar}$ are uniformly bounded in $C([0, T_0]; H^{k})$.

Next, we solve equation \eqref{S1eq5} via \eqref{sys-u-a} and $u^{\hbar}= \pa_x S^{\hbar}$.  Then it suffices to prove that  $\int_{\mathbb{T}} S^{\hbar}(t)$ are uniformly bounded on $[0, T_0]$.  We know from  \eqref{S1eq5} that
\begin{equation}
 \pa_t S^\hbar+\frac12 (u^{\hbar})^2+f^\hbar(t, x)+|a^\hbar|^2-1=0.   \notag
\end{equation}
Since $\int_{\mathbb{T}} f^{\hbar}(t)=0$ and $\int_{\mathbb{T}} \frac12 (u^{\hbar})^2+ |a^\hbar|^2-1$ are uniformly bounded, this concludes the proof.\\

\noindent$\bullet$ {\bf Step 2.}  \textit{For $\hbar$ small enough,  the solutions $(a^{\hbar}, S^{\hbar})$ are uniformly bounded in $C([0, T]; H^{k}\times H^{k+1})$.  Moreover, $(a^{\hbar}, S^{\hbar}) \textrm{ tends to } (a_0, S_0) \textrm{  in } C([0, T]; H^k\times H^{k+1})$ as $\hbar\rightarrow 0^+$.}

It is known from the assumption of Theorem \ref{thm-sem-main} that \eqref{S1eq6a} admits a solution $(\r_0, u_0)$ in $C([0, T]; H^{k})$.   At first, by denoting $w^T(t):= \left(a^r_0(t), a^i_0(t), u_0(t)\right)$  we prove that  the zeroth order  limit system \eqref{S1eq6b},
\begin{gather}\label{w-lim}
\pa_t w+ A(w) \pa_x w=  f_0,\\
w(0)=\displaystyle (a_0^{\rm r},a_0^{\rm i},u_0)|_{t=0}=({\rm Re}(a_0^0), {\rm Im}(a_0^0),\pa_xS),
\end{gather}
 has a solution $w_0\in C([0, T]; H^{k+2})$.  Indeed, suppose that the maximal solution is on the interval $[0, s)$ with $s<T$, then $\r_0, u_0$ are bounded in $ C([0, s]; H^{k+2})$.  We know from  \eqref{S1eq6b} that $a_0$ belongs to $C([0, s]; H^{k+1})$, which is in contradiction of the definition of $s$.

 Let us define $v^{\hbar}= w^{\hbar}- w_0$. Then comparing \eqref{w-lim} and \eqref{eq-w-im} we get
 \begin{gather}\label{v-hbar}
 \pa_t v^{\hbar}+ \big( A(v^{\hbar})+ A(w_0)\big) \pa_x v^{\hbar}+ A(v^{\hbar})\pa_x w_0= \hbar \big( H(w_0)+ H(v^{\hbar})+ f_1\big),\\
v^{\hbar}(0)= \left({\rm Re}(a^0(\hbar))- {\rm Re} (a_0^0),  {\rm Im}(a^0(\hbar))- {\rm Im}(a_0^0), 0\right),
 \end{gather}
As $w_0$ belongs to  $C([0, T]; H^{k+2})$, we get that
 \begin{equation}
 \mid (\hbar H(\pa_x^{\alpha} w), \pa_x^{\alpha} v^{\hbar}) \mid \leq \hbar C  \lVert w_0\lVert_{H^{k+2}} \lVert v^{\hbar}\lVert_{H^{k}},  \forall \alpha\leq k,    \notag
 \end{equation}
which together with similar energy estimates as in the previous step lead to
\begin{equation}
\pa_t \sum_{\alpha \leq k} (A_0 \pa_x^{\alpha} v^{\hbar}, \pa_x^{\alpha} v^{\hbar}) \leq C  \lVert v^{h} \lVert_{H^k}^2 ( \lVert v^{h} \lVert_{C^1} +1) +  \hbar C \lVert v^{\hbar}\lVert_{H^{k}} + \hbar C \lVert v^{h} \lVert_{H^k}.    \notag
\end{equation}
Since $\lVert v^{\hbar}(0)\lVert_{H^k}$ tends to 0 as $\hbar$ tends to 0,  we are able to find positive $\hbar_0$ such that
\begin{equation}\label{v-h-bound}
\lVert v^{\hbar}\lVert_{C([0, T]; H^k)}\leq C \hbar+ C \lVert v^{\hbar}(0)\lVert_{H^k}\leq C \hbar, \forall 0< \hbar< \hbar_0.    \notag
\end{equation}
Therefore,  the solutions $(a^h, u^h)$ converge uniformly to $(a_0, u_0)$ in $C([0, T]; H^k \times H^k)$ as $\hbar \rightarrow 0^+$.  Meanwhile we know that $S^h$ (resp. $S_0$) satisfies \eqref{S1eq5} (resp. \eqref{S1eq6}), thus direct calculations as above show that
\begin{equation}
\int_{\mathbb{T}} S^{\hbar}(t)\rightarrow \int_{\mathbb{T}} S_0(t) \textrm{ in } C([0, T])\textrm{ as } \hbar\rightarrow 0^+.  \notag
\end{equation}

Hence
\begin{equation}
(a^{\hbar}, S^{\hbar}) \textrm{ tends to } (a_0, S_0) \textrm{  in } C([0, T]; H^k\times H^{k+1}) \textrm{ as } \hbar\rightarrow 0^+.  \notag
\end{equation}
\\
\noindent$\bullet$ {\bf Step 3.} \textit{Let $\tilde{S}^{\hbar}= (S^{\hbar}-S_0)/\hbar, \tilde{v}^{\hbar}= v^{\hbar}/\hbar$. Then $(\tilde{a}^{\hbar}, \tilde{S}^{\hbar})$ tends to $(a_1, S_1)$ in $C([0, T]; H^{k-2}\times H^{k-1})$.}

We deduce from \eqref{v-hbar} and the last equation of \eqref{sys-u-a} that
\begin{gather}\label{v-til-1}
\pa_t \tilde{v}^{\hbar}+ \big( \hbar A(\tilde{v}^{\hbar})+ A(w_0) \big) \pa_x \tilde{v}^{\hbar}+ A(\tilde{v}^{\hbar}) \pa_x w_0= H(w_0)+ \hbar H(\tilde{v}^{\hbar})+ f_1,\\
\tilde{v}^{\hbar}(0)= (\rm{Re} (a^0_1+ R(\hbar)), \rm{Im} (a^0_1+ R(\hbar)), 0).\label{v-til-2}
\end{gather}
The limit system of \eqref{v-til-1}--\eqref{v-til-2} reads
\begin{gather}\label{v-til-1-lim}
\pa_t \tilde{v}+ A(w_0)  \pa_x \tilde{v}+ A(\tilde{v}^{\hbar}) \pa_x w_0= H(w_0)+ f_1,\\
\tilde{v}(0)= (\rm{Re } (a^0_1), \rm{Im } (a^0_1), 0),\label{v-til-2-lim}
\end{gather}
which is exactly the first order system \eqref{S1eq7b}. Let us denote the solution of this linear system by $v_1=(a_1^r, a_1^i, u_1)$.

Thanks to \eqref{v-h-bound}, $\tilde{v}^{\hbar}$ are uniformly bounded in $C([0, T]; H^{k})\subset L^{\infty}([0, T]; H^k)$. Then by plugging this bound into \eqref{v-til-1}--\eqref{v-til-2},  we obtain the uniform boundedness of  $\pa_t \tilde{v}^{\hbar}$  in $C([0, T]; H^{k-2})\subset L^{\infty}([0, T]; H^{k-2})$.  Therefore, up to  a subsequence,  $\tilde{v}^{\hbar}$ converges to some function $v'$ in $C([0, T]; H^{k-2})$.  By taking the limit, we find that $v'$ solves the equation \eqref{v-til-1-lim}--\eqref{v-til-2-lim}. Hence $v'$ coincidence with $v_1$.  To this end, we deduce from
\begin{equation}
\pa_t \tilde{S}^{\hbar}+ u_0 \tilde{u}^{\hbar}+ \frac{1}{2} \hbar (\tilde{u}^{\hbar})^2+ F_2+ (\tilde{a}^{\hbar} \bar{a}_0+ \bar{\tilde{a}}^{\hbar} a_0)+ \hbar \mid \tilde{a}^{\hbar} \mid^2=0, \notag
\end{equation}
\begin{equation}
\pa_t S_1+ u_0 u_1+ 2\textrm{Re}(a_0\bar{a}_1)+ F_2=0, \notag
\end{equation}
that
\begin{equation}
\int_{\mathbb{T}} \tilde{S}^{\hbar}(t)\rightarrow \int_{\mathbb{T}} S_1(t) \textrm{ in } C([0, T]) \textrm{ as } \hbar\rightarrow 0^+.  \notag
\end{equation}
Therefore
\begin{equation}
(\tilde{a}^{\hbar}, \tilde{S}^{\hbar}) \textrm{ tends to } (a_1, S_1) \textrm{  in } C([0, T]; H^{k-2}\times H^{k-1}) \textrm{ as } \hbar\rightarrow 0^+.  \notag
\end{equation}
More precisely, we get the first order expansion of $\psi^{\hbar}$:
\begin{gather*}
a^{\hbar}= a_0+ \hbar a_1+ \hbar r_a^{\hbar}, \;\;
S^{\hbar}= S_0+ \hbar S_1+ \hbar r_S^{\hbar},  \\
\lim_{\hbar\rightarrow 0^+} \lVert r_a^{\hbar} \lVert_{C([0, T]; H^{k-2})}=0, \;\; \lim_{\hbar\rightarrow 0^+} \lVert r_S^{\hbar} \lVert_{C([0, T]; H^{k-1})}=0,
\end{gather*}
and complete the proof of Theorem \ref{thm-sem-main}. \end{proof}

\section*{Acknowledgments}

 Part of this work was done when Shengquan Xiang was visiting the  Academy of Mathematics and Systems Sciences (CAS), and when Ping Zhang was visiting Jacques-Louis Lions Laboratory of Sorbonne Universit\'{e}.
   We appreciate the hospitality and the financial support of these institutions.   Shengquan Xiang is financially supported by the Chair of Partial Differential Equations at EPFL. Ping Zhang is partially supported by K.C.Wong Education Foundation and NSF of China under Grants   11731007, 12031006 and 11688101.

\bibliographystyle{plain}
\bibliography{sch}

\begin{thebibliography}{10}

\bibitem{Agrachev-Sarychev-2005-JMFM}
Andrey~A. Agrachev and Andrey~V. Sarychev.
\newblock Navier-{S}tokes equations: controllability by means of low modes
  forcing.
\newblock {\em J. Math. Fluid Mech.}, 7(1):108--152, 2005.

\bibitem{Agrachev-Sarychev-2006-CMP}
Andrey~A. Agrachev and Andrey~V. Sarychev.
\newblock Controllability of 2{D} {E}uler and {N}avier-{S}tokes equations by
  degenerate forcing.
\newblock {\em Comm. Math. Phys.}, 265(3):673--697, 2006.

\bibitem{AC09}
Thomas Alazard and R\'{e}mi Carles.
\newblock Supercritical geometric optics for nonlinear {S}chr\"{o}dinger
  equations.
\newblock {\em Arch. Ration. Mech. Anal.}, 194(1):315--347, 2009.

\bibitem{Beauchard-2005-JMPA9}
Karine Beauchard.
\newblock Local controllability of a 1-{D} {S}chr\"{o}dinger equation.
\newblock {\em J. Math. Pures Appl. (9)}, 84(7):851--956, 2005.

\bibitem{Beauchard-Laurent-2010}
Karine Beauchard and Camille Laurent.
\newblock Local controllability of 1{D} linear and nonlinear {S}chr\"{o}dinger
  equations with bilinear control.
\newblock {\em J. Math. Pures Appl. (9)}, 94(5):520--554, 2010.

\bibitem{1993-Beirao-da-Veiga-CPAM}
H.~Beir\~{a}o~da Veiga.
\newblock Perturbation theorems for linear hyperbolic mixed problems and
  applications to the compressible {E}uler equations.
\newblock {\em Comm. Pure Appl. Math.}, 46(2):221--259, 1993.

\bibitem{coron-1996-euler-return}
Jean-Michel Coron.
\newblock On the controllability of {$2$}-{D} incompressible perfect fluids.
\newblock {\em J. Math. Pures Appl. (9)}, 75(2):155--188, 1996.

\bibitem{116}
Jean-Michel Coron.
\newblock Local controllability of a 1-{D} tank containing a fluid modeled by
  the shallow water equations.
\newblock {\em ESAIM Control Optim. Calc. Var.}, 8:513--554 (electronic), 2002.
\newblock A tribute to J. L. Lions.

\bibitem{coron}
Jean-Michel Coron.
\newblock {\em Control and nonlinearity}, volume 136 of {\em Mathematical
  Surveys and Monographs}.
\newblock American Mathematical Society, Providence, RI, 2007.

\bibitem{1996-Coron-Fursikov-RJMP}
Jean-Michel Coron and Andrei~V. Fursikov.
\newblock Global exact controllability of the {$2$}{D} {N}avier-{S}tokes
  equations on a manifold without boundary.
\newblock {\em Russian J. Math. Phys.}, 4(4):429--448, 1996.

\bibitem{coron:hal-03161523}
Jean-Michel. Coron, Amaury Hayat, Shengquan Xiang, and Christophe Zhang.
\newblock {Stabilization of the linearized water tank system}.
\newblock Preprint, hal-03161523, March 2021.

\bibitem{Nersesyan-2021}
Alessandro Duca and Vahagan Nersesyan.
\newblock Bilinear control and growth of sobolev norms for the nonlinear
  {S}chrödinger equation.
\newblock {\em Preprint, arXiv:2101.12103}, 2021.

\bibitem{MR2103189}
E.~Fern\'{a}ndez-Cara, S.~Guerrero, O.~Yu. Imanuvilov, and J.-P. Puel.
\newblock Local exact controllability of the {N}avier-{S}tokes system.
\newblock {\em J. Math. Pures Appl. (9)}, 83(12):1501--1542, 2004.

\bibitem{2000-Glass-COCV}
Olivier Glass.
\newblock Exact boundary controllability of 3-{D} {E}uler equation.
\newblock {\em ESAIM Control Optim. Calc. Var.}, 5:1--44, 2000.

\bibitem{Glass-2007-JEMS}
Olivier Glass.
\newblock On the controllability of the 1-{D} isentropic {E}uler equation.
\newblock {\em J. Eur. Math. Soc. (JEMS)}, 9(3):427--486, 2007.

\bibitem{1998-Grenier-PAMS}
Emmanuel Grenier.
\newblock Semiclassical limit of the nonlinear {S}chr\"{o}dinger equation in
  small time.
\newblock {\em Proc. Amer. Math. Soc.}, 126(2):523--530, 1998.

\bibitem{SaintVenantPI}
Amaury Hayat.
\newblock {PI controllers for the general Saint-Venant equations}.
\newblock 2021.
\newblock Preprint, arXiv:2108.02703.

\bibitem{Krieger-Xiang-2020}
Joachim Krieger and Shengquan Xiang.
\newblock Boundary stabilization of focusing {N}{L}{K}{G} near unstable
  equilibria: radial case.
\newblock {\em Preprint, arXiv:1911.03661}, 2020.

\bibitem{Laurent-2014}
Camille Laurent.
\newblock Internal control of the {S}chr\"{o}dinger equation.
\newblock {\em Math. Control Relat. Fields}, 4(2):161--186, 2014.

\bibitem{Lebeau-1992}
Gilles Lebeau.
\newblock Contr\^{o}le de l'\'{e}quation de {S}chr\"{o}dinger.
\newblock {\em J. Math. Pures Appl. (9)}, 71(3):267--291, 1992.

\bibitem{2006-Lin-Zhang-ARMA}
Fanghua Lin and Ping Zhang.
\newblock Semiclassical limit of the {G}ross-{P}itaevskii equation in an
  exterior domain.
\newblock {\em Arch. Ration. Mech. Anal.}, 179(1):79--107, 2006.

\bibitem{1927-Madelung-ZP}
Erwin Madelung.
\newblock Quantentheorie in hydrodynamischer form.
\newblock {\em Zeit. f. Phys.}, 40(3--4):322–326, 1927.

\bibitem{2010-Nersisyan-COCV}
Hayk Nersisyan.
\newblock Controllability of 3{D} incompressible {E}uler equations by a
  finite-dimensional external force.
\newblock {\em ESAIM Control Optim. Calc. Var.}, 16(3):677--694, 2010.

\bibitem{2011-Nersisyan-CPDE}
Hayk Nersisyan.
\newblock Controllability of the 3{D} compressible {E}uler system.
\newblock {\em Comm. Partial Differential Equations}, 36(9):1544--1564, 2011.

\bibitem{2012-Sarychev-MCRF}
Andrey Sarychev.
\newblock Controllability of the cubic {S}chroedinger equation via a
  low-dimensional source term.
\newblock {\em Math. Control Relat. Fields}, 2(3):247--270, 2012.

\bibitem{2006-Shirikyan-CMP}
Armen Shirikyan.
\newblock Approximate controllability of three-dimensional {N}avier-{S}tokes
  equations.
\newblock {\em Comm. Math. Phys.}, 266(1):123--151, 2006.

\bibitem{2018-Shirikyan-PAFA}
Armen Shirikyan.
\newblock Control theory for the {B}urgers equation: {A}grachev-{S}arychev
  approach.
\newblock {\em Pure Appl. Funct. Anal.}, 3(1):219--240, 2018.

\bibitem{Xiang-NS-2020}
Shengquan Xiang.
\newblock Small-time local stabilization of the two dimensional incompressible
  {N}avier-{S}tokes equations.
\newblock {\em Preprint, arXiv:2010.13696}, 2020.

\bibitem{2002-Zhang-JPDE}
Ping Zhang.
\newblock Semiclassical limit of nonlinear {S}chr\"{o}dinger equation. {II}.
\newblock {\em J. Partial Differential Equations}, 15(2):83--96, 2002.

\bibitem{2002-Zhang-SIAMJMA}
Ping Zhang.
\newblock Wigner measure and the semiclassical limit of
  {S}chr\"{o}dinger-{P}oisson equations.
\newblock {\em SIAM J. Math. Anal.}, 34(3):700--718, 2002.

\end{thebibliography}

\end{document}